\documentclass[final,3p,times]{elsarticle}
\usepackage{microtype}
\usepackage[english]{babel}
\usepackage{amsfonts,amsmath}
\usepackage{color}

\newtheorem{thm}{Theorem}
\newtheorem{defin}{Definition}
\newtheorem{lem}[thm]{Lemma}
\newtheorem{coro}[thm]{Corollary}
\newtheorem{remark}[thm]{Remark}
\newtheorem{exa}{Example}
\newenvironment{example}{\begin{exa} \rm }{\hfill$\Box$ \end{exa}}
\newenvironment{definition}{\begin{defin} \rm }{\end{defin}}
\newenvironment{proof}{\begin{trivlist}\item[] \mbox{\it Proof. }}
{\hfill$\Box$ \end{trivlist}}
\def\caP{{\cal P}(N)} 
\def\conv{\mbox{\rm conv}\,} 
\def\lin{\mbox{\rm Lin}\,} 
\def\ext{\mbox{\sf ext}\,} 
\def\inter{\mbox{\rm relint}\,} 
\def\ci{\perp\!\!\!\perp} 
\def\dv{\mathbb} 
\def\inc{\chi} 
\def\incS{\chi} 
\def\lov{\mbox{\l}} 
\def\ros{\hbar} 
\def\cor{C}
\def\web{W} 
\def\stagame{{\cal G}(N)} 
\def\supmogame{{\cal G}_{\diamond}(N)} 
\def\submogame{\bar{{\cal R}}(N)} 
\def\komstagame{\bar{{\cal R}}(N)} 
\def\submogame{\bar{{\cal R}}_{\circ}(N)} 
\def\eletri{{\cal E}(N)} 
\def\supermo{\lozenge (N)} 
\def\compo{} 
\def\facelatt{{\cal F}(N)} 
\def\VR{\kern-\arraycolsep\strut\vrule &\kern-\arraycolsep} 
\def\vr{\kern-\arraycolsep & \kern-\arraycolsep} 
\makeatletter
\def\bbordermatrix#1{\begingroup \m@th
  \@tempdima 4.75\p@
  \setbox\z@\vbox{%
    \def\cr{\crcr\noalign{\kern2\p@\global\let\cr\endline}}%
    \ialign{$##$\hfil\kern2\p@\kern\@tempdima&\thinspace\hfil$##$\hfil
      &&\quad\hfil$##$\hfil\crcr
      \omit\strut\hfil\crcr\noalign{\kern-\baselineskip}%
      #1\crcr\omit\strut\cr}}%
  \setbox\tw@\vbox{\unvcopy\z@\global\setbox\@ne\lastbox}%
  \setbox\tw@\hbox{\unhbox\@ne\unskip\global\setbox\@ne\lastbox}%
  \setbox\tw@\hbox{$\kern\wd\@ne\kern-\@tempdima\left[\kern-\wd\@ne
    \global\setbox\@ne\vbox{\box\@ne\kern2\p@}%
    \vcenter{\kern-\ht\@ne\unvbox\z@\kern-\baselineskip}\,\right]$}%
  \null\;\vbox{\kern\ht\@ne\box\tw@}\endgroup}
\makeatother

\journal{Discrete Applied Mathematics}
\begin{document}

\title{Core-based criterion for extreme supermodular functions}
\author[utia]{Milan Studen\'{y}\corref{cor}}
\ead{studeny@utia.cas.cz}
\author[utia,mi]{Tom\'{a}\v{s} Kroupa}
\ead{kroupa@utia.cas.cz}
\address[utia]{Institute of Information Theory and Automation, Czech Academy of Sciences\\
Pod Vod\'{a}renskou v\v{e}\v{z}\'{\i} 4, 182 08 Prague,  Czech Republic}
\address[mi]{Dipartimento di Matematica ``Federigo Enriques'', Universit\`a degli Studi di Milano,\\ Via Cesare Saldini 50, 20133 Milano, Italy}
\cortext[cor]{Corresponding author}

\begin{abstract}
We give a necessary and sufficient condition for extremality of a supermodular function
based on its min-representation by means of (vertices of) the corresponding core polytope.
The condition leads to solving a certain simple linear equation system determined by the combinatorial core structure.
This result allows us to characterize indecomposability in the class of generalized permutohedra.
We provide an in-depth comparison between our result and the description of extremality in the supermodular/submodular
cone achieved by other researchers.
\end{abstract}

\begin{keyword}
supermodular function \sep submodular function \sep core \sep conditional independence \sep generalized permutohedron \sep indecomposable polytope
\end{keyword}

\maketitle

\section{Introduction}
Supermodular functions have been investigated in various branches of discrete mathematics, namely
in connection with cooperative games \cite{sha72}, conditional independence structures \cite{stu05} and generalized permutohedra \cite{post05}.
Submodular functions, their mirror images, were studied in matroid theory \cite{oxl92} and combinatorial optimization \cite{fuj91,sch03}.
Throughout this paper we regard a supermodular function as a real function defined on the power set of a finite set
of variables and satisfying the supermodularity law.
As the set of (suitably standardized) supermodular functions forms a pointed polyhedral cone in a finite-dimensional space,
it has a finite number of extreme rays. Characterizing extremality in the supermodular cone is of vital importance
for understanding its structure. This task is solved in this paper: our main result, Theorem \ref{thm.1},
provides a~necessary and sufficient condition for extremality of a supermodular function.
The condition has the form of a simple criterion based on solving a system of linear equations.

The research on extreme supermodular functions has been ongoing in a number of different mathematical disciplines.
Let us mention just a few of them to summarize the motivation for this paper and to recall some previous results related to
the supermodular/submodular cone. Our list is by no means exhaustive.
\begin{enumerate}
\item \emph{Coalition games.} The mathematical model of a cooperative game in a coalitional form is
due to von Neumann and Morgenstern \cite{vNM44}. \emph{Convex games} were introduced as supermodular functions on
the class of all coalitions by Shapley \cite{sha72}. Interestingly enough, in the 1972 paper Shapley enumerates
all the extreme rays of the cone of convex games over the four-player set. Nonetheless, he claims that
``For larger~$n$, little is known about the set of all extremals''.
A lot of effort was exerted to describe the geometrical structure of the \emph{core}, which is a non-empty polytope
associated with any (convex) game. The core concept is also among the crucial instruments employed in this paper.
Namely we rely on the characterization of the vertices of the core achieved by Shapley \cite{sha72} and Weber \cite{web88}.
The properties of the core allow one to characterize convex games within the set of all the coalition games; see \cite{ich81,DK00}.
Kuipers et al.\ \cite{KVV10} provided a facet description of the supermodular cone.
Danilov and Koshevoy \cite{DK00} employed the M\"{o}bius inversion in order to express the core of a
(not necessarily supermodular) coalition game as a signed Minkowski sum of standard simplices.
\item \emph{Conditional independence structures.} Conditional independence structures arising
in discrete probabilistic framework belong to a wider class of {\em structural\/} (conditional)
{\em independence models}, which can be interpreted as models produced by supermodular functions; see \cite[\S\,5.4.2]{stu05}.
In fact, the lattice of structural independence \mbox{models} is anti-isomorphic to the face lattice of
the cone of supermodular functions. Thus, characterizing the extreme rays of the cone of standardized
supermodular functions can have both the theoretical significance in characterizing co-atoms of the
lattice of structural models and some practical consequences for conditional independence implication; see \cite[\S\,4.1]{BHLS}.
Extreme supermodular functions also establish quite an important class of inequalities that are used in
(integer) linear programming approach to learning Bayesian network structure; see
\cite[\S\,3.1]{SV11} or \cite[\S\,7]{CHS16}.
There were attempts to classify extreme supermodular functions and the operations with them \cite{SBK00,KSTT12};
see also the open problems from \S\,9.1.2 of \cite{stu05}.
\item \emph{Generalized permutohedra.} These polytopes were introduced by Postnikov \cite{post05,post08}
as the polytopes obtainable by moving vertices of the usual permutohedron while the directions of edges are preserved.
The connection of generalized permutohedra to supermodular and submodular functions has been indicated by Doker \cite{dok11}.
Morton \cite{mort07} earlier discussed the role of generalized permutohedra directly in the study of conditional independence structures.
The class of generalized permutohedra appears to coincide with the class of core polytopes for supermodular functions.
This allows us to derive as a by-product of our result a necessary and sufficient condition for a generalized
permutohedron to be {\em indecomposable} in the sense of Meyer \cite{mey74}. Although the task to characterize
indecomposable generalized permutohedra has not been raised in the literature, we hope it is relevant to this topic.
\item \emph{Combinatorial optimization and matroids.}
The importance of {\em submodular functions}, which can be viewed as mirror images of supermodular functions,
has widely been recognized in combinatorial optimization; see \cite{fuj91}, for example.
In fact, the core polytopes correspond to the so-called {\em base polyhedra} for submodular functions. In this context,
a non-decreasing submodular function is called a \emph{rank function} of a polymatroid. As noted by Schrijver \cite[p.\ 781]{sch03},
already Edmonds \cite{edm70} raised the problem of determining the extreme rays of the cone of rank functions of polymatroids.
Nguyen \cite{ngu78} gave a criterion to recognize whether a rank function of a {\em matroid} \cite{oxl92}
generates an extreme rays of that cone.
One of his followers was Kashiwabara \cite{kas00} who provided more general sufficient conditions for extremality
of certain integer-valued submodular functions in terms of their combinatorial properties.
A few other researchers studied the submodular functions in different frameworks. These functions have wide
applications in computer science as explained by \v{Z}ivn\'{y} et al.\ in \cite[\S\,1.3]{Zivny09}, who also
discussed a conjecture on the extreme rays of the cone of Boolean submodular functions, which
was raised in supermodular context by Promislow and Young \cite{PY05}.

\item \emph{Imprecise probabilities.} Theory of imprecise probabilities  deals with generalized models
of uncertainty reaching beyond the usual assumption in probability theory, namely the additivity axiom.
One of the basic concepts in this theory is that of a {\em coherent lower probability},
which corresponds to a well-known game-theoretical notion of an {\em exact game}; see Corollary 3.3.4 in the book \cite{Walley91} by Walley.
Similarly, the concepts of a {\em credal set\/} and of a {\em 2-monotone lower probability} are the counterparts of
the concepts of a~core polytope and of a (normalized) non-negative supermodular game, respectively.
Quaeghebeur and de~Cooman \cite{QdC08} raised the question of characterizing the extreme lower probabilities
and computed some of them for a small number of variables. Even a more general task has been addressed in the literature:
the characterization of {\em extreme lower previsions\/} given by De~Bock and de~Cooman \cite{dBC15} relates them
to indecomposable compact convex sets in a~finite-dimensional space.
\end{enumerate}

In this study we proceed without having any particular domain of application in mind, but being aware of
the presence of this topic on the crossroad of many different disciplines mentioned above.
We make an ample use of techniques and results from coalition game theory and finite-dimensional convex geometry.
The key technical tool presented herein is a transformation
which associates a certain polytope, called the {\em Weber set}, with every game.
The point is that the Weber set of a supermodular game coincides with its core as defined in coalition game theory.
Our main result, Theorem \ref{thm.1}, basically asserts that a supermodular game is extreme if and only if
the combinatorial structure of its core fully determines its geometry.
The combinatorial concept of a ``core structure" we use here has already appeared in the literature: it was formally defined by Kuipers et al.\/ \cite{KVV10}.

The close relation between supermodular functions and generalized permutohedra pervades this paper. This correspondence is realized via a min-representation of a~supermodular function based on its core. Our Corollary \ref{cor.GP-supermod-core} shows that the cores of supermodular games
coincide with generalized permutohedra. As a consequence of the extremality characterization
we provide a necessary and sufficient condition for indecomposability of generalized permutohedra (see Theorem \ref{thm.ired}).

The article is structured as follows. We fix our notation and terminology in \S\,\ref{sec.prelimi}.
In particular, we introduce the key notions of payoff-array transformation and the Weber set there.
Moreover, we formulate fundamental Lemma~\ref{lem.1} and explain how to recover the vertices of the core for a supermodular game.
Our main result, Theorem \ref{thm.1}, is formulated in \S\,\ref{sec.characterization}.
The use of the main theorem is demonstrated by some examples and
the interpretation of our criterion of extremality is discussed.
The proof of Theorem \ref{thm.1} is postponed to \S\,\ref{sec.proof}.
A~close connection between supermodular games and generalized permutohedra is revealed in \S\,\ref{sec.permutohedra}.
Sections \S\,\ref{sec.RW} and \S\,\ref{sec.remark-other} contain an extensive and detailed discussion on previous results on extremality criteria for supermodular and submodular functions from the literature.
In order to show that our criterion is indeed new, although perhaps analogous in certain aspects to previous criteria,
we analyze the results by Rosenm\"{u}ller and Weidner \cite{RW73,RW74} and by Nguyen \cite{ngu78}.
We conclude the main part of the paper with an~outlook towards further research  in \S\,\ref{sec.concl},
where we also formulate an open problem to characterize the cone of exact games.
Supermodularity of a set function has many equivalent formulations: they are summarized in \ref{sec.apex-supermod}.
In \ref{sec.apex-CI} we explain the significance of supermodular functions in the context of conditional independence
structures. In particular, we show that the face lattice of the supermodular cone is anti-isomorphic to the lattice
of structural independence models.

\section{Preliminaries}\label{sec.prelimi}
We introduce our notation and recall basic concepts in this section.

\subsection{Notation and some basic terminology}\label{ssec.notation}
Let $N$ be a finite non-empty set of {\em variables\,};\footnote{See \ref{sec.apex-CI}, Remark  \ref{rem.why-variable},
for the motivation of our terminology.}
$n:= |N|\geq 2$ and $\caP := \{ A : A\subseteq N\}$.
In cooperative game theory, variables correspond to players, subsets of $N$ are coalitions.
Intentionally, no reference total ordering on the set of variables $N$ is fixed
to avoid possible later misinterpretation. Thus, we regard $N$ as an un-ordered set,
for example $N=\{ a,b,c\}$.

The symbol ${\dv R}^{N}$ will denote the vector space of real $N$-tuples, that is, real vectors with components indexed by $N$;
these are formally mappings from $N$ to the real line ${\dv R}$.
For every set $S\subseteq N$, the {\em incidence vector\/} of $S$ is a vector in ${\dv R}^{N}$ with the coordinates
\begin{equation}
\incS_{S}(j) =\left\{
\begin{array}{cl}
1 & \mbox{if $j\in S$},\\
0 & \mbox{if $j\in N\setminus S$},
\end{array}
\right.
\qquad \mbox{for any $j\in N$.}
\label{eq.incidence}
\end{equation}
Given $v\in {\dv R}^{N}$ and $j\in N$, we will sometimes, when it appears to be convenient,
write $v_{j}$ instead of $v(j)$. Thus, a vector $v\in {\dv R}^{N}$ may alternatively
by written as $[v_{i}]_{i\in N}$.

Occasionally, the symbol $i$ for $i\in N$ will be used as a shorthand for the singleton~$\{ i\}$. Therefore,
$\inc_{i}\equiv\incS_{\{ i\}}\in {\dv R}^{N}$ will denote the zero-one identifier of the variable $i\in N$.
By a~{\em polytope} (in ${\dv R}^{N}$) we mean the convex hull of finitely many points in ${\dv R}^{N}$. The {\em Minkowski sum}
of polytopes $P,Q\subseteq {\dv R}^{N}$ is defined by
$$
P\oplus Q ~:=~ \{\, x+y\in {\dv R}^{N}\,:\ x\in P ~\& ~ y\in Q\,\}.
$$
For every $\emptyset\neq S\subseteq N$, the symbol $\Delta_{S}$ will denote the corresponding {\em standard simplex\/} in ${\dv R}^{N}$, which is the polytope of the form
$\Delta_{S}:= \conv (\{ \inc_{i} :\ i\in S\})$.\footnote{The symbol $\conv (Q)$ denotes the convex hull of $Q\subseteq {\dv R}^{N}$.}

We will also deal with real functions of coalitions, which form the vector space ${\dv R}^{\caP}$,
formally defined as the class of mappings from the power set $\caP$ to ${\dv R}$.
Given a set $S\subseteq N$, the corresponding standard basis vector in ${\dv R}^{\caP}$ will be denoted as follows:
\begin{equation}\label{def:identifier}
\delta_{S}(A) =\left\{
\begin{array}{cl}
1 & \mbox{if $A=S$},\\
0 & \mbox{if $A\neq S$},
\end{array}
\right.
\qquad \mbox{for any $A\subseteq N$.}
\end{equation}
This notation simplifies some formulas for elements in ${\dv R}^{\caP}$. For example, for $i\in N$,
we introduce a special notation for the identifier of supersets of $\{i\}$:
\begin{equation}\label{def:upper}
m^{\uparrow i} \,:= \sum_{S\subseteq N\,:\ i\in S} \delta_{S}\,,
~~ \mbox{that is,}\quad
m^{\uparrow i}(T)=
\left\{
\begin{array}{lc}
1 & \mbox{if $i\in T$},\\
0 & \mbox{if $i\not\in T$},
\end{array}
\right.
~\mbox{for $T\subseteq N$}.
\end{equation}

\begin{defin}[game, core, supermodular game]\rm ~\\
By a {\em game\/} over $N$, which is our shorthand for a cooperative transferable utility game (see \cite{LR89,vNM44}, for example),
we will understand a mapping $m:\caP \to {\dv R}$
satisfying $m(\emptyset )=0$.
The {\em core\/} of a game $m$ is a polytope in ${\dv R}^{N}$, defined by
\begin{equation}
\cor (m) := \{\, [v_{i}]_{i\in N}\in {\dv R}^{N}\, : ~ \sum_{i\in N} v_{i}= m(N) \,~\&~\,
\forall\, S\subseteq N ~\sum_{i\in S} v_{i}\geq m(S)\,\}\,.
\label{eq.def.core}
\end{equation}
A game $m$ is {\em balanced\/} if $\cor (m)\neq\emptyset$. A balanced game $m$ is called {\em exact\/} if
$$
\forall\, S\subseteq N\qquad m(S)=\min_{v\in\cor (m)}\, \sum_{i\in S} v_{i}\,.
$$
A set function $m\in {\dv R}^{\caP}$  is {\em supermodular\/} if
\begin{equation}
\forall\, A,B\subseteq N\quad
m(A)+m(B)\leq m(A\cup B)+ m(A\cap B)\,.
\label{eq.supermodul}
\end{equation}
A game $m$ over $N$ will be called {\em standardized\/} if $m(S)=0$ for $S\subseteq N$, $|S|\leq 1$.
\end{defin}
\noindent
A well-known fact is that every supermodular game is exact and thus necessarily balanced; see \cite[\S\,5]{Schmeidler72}.
A set function $r\in {\dv R}^{\caP}$ is {\em submodular} if $-r$ is supermodular.
A {\em modular\/} set function is a set function which is simultaneously supermodular and submodular. An easy
observation is that the linear space of modular set functions in ${\dv R}^{\caP}$ has the dimension
$n+1$, with a linear basis consisting of a non-zero constant set function and
$\{ m^{\uparrow i}:\, i\in N\}$. The dimension of the linear space of modular games over $N$ is $n$
since the only constant modular game is the zero function.

We also introduce a special notation for several sets of games:
\begin{eqnarray*}
\supermo && \text{is the set of all supermodular games over $N$,}\\
\stagame  && \mbox{is the set of all standardized games over $N$, and}\\
\supmogame && \mbox{is the set of all supermodular standardized games over $N$.}
\end{eqnarray*}
The supermodular cone $\supermo$ is not pointed because it contains the linear subspace of all modular games.
Therefore, we introduce a ``standardization'' procedure which maps $\supermo$ linearly onto the pointed cone $\supmogame$.
Given $m\in\supermo$ we put
\begin{equation}
m^{\star}(S) := m(S)-\sum_{i\in S}\, m(\{i\}) \quad \mbox{for $S\subseteq N$,~~ that is, ~~}
m^{\star}=m-\sum_{i\in N}\, m(\{ i\})\cdot m^{\uparrow i},
\label{eq.standard}
\end{equation}
and observe $m^{\star}\in\supmogame$. Note that this is just one of possible ways to standardize
supermodular functions; see Remark~5.3 in \cite{stu05} for further options.
Since the only standardized modular game is the zero function, $m^{\star}\in\supmogame$ given by
\eqref{eq.standard} is unique such that $m=m^{\star}+g$ for a modular game $g$.

\subsection{Weber set and a fundamental lemma}\label{ssec.payoff-trans}

A crucial technical tool in the proof of our main result is a certain linear transformation
defined here, which is related to the game-theoretical concept of the Weber set \cite{web88}.
\smallskip

Let us denote by $\Upsilon$ the set of all {\em enumerations\/} of
elements in $N$, introduced formally as bijections
$\pi \colon \{ 1,\ldots ,n\} \to N$ from the ordered set
$\{ 1,\ldots ,n\}$ onto $N$. Elements of $\Upsilon$ are in
a one-to-one correspondence with {\em permutations\/} on $N$: provided a particular
distinguished ``reference" enumeration $\upsilon$ is chosen and fixed, any $\pi\in\Upsilon$
is a composition of a uniquely determined permutation on $N$ with $\upsilon$. Alternatively,
elements of $\Upsilon$ can be described by permutations on $\{ 1,\ldots ,n\}$: any
$\pi\in\Upsilon$ is a composition of the reference enumeration $\upsilon$ with a unique permutation on $\{ 1,\ldots ,n\}$.
We intentionally regard $N$ as an un-ordered set, unlike some other authors
who identify $N$ with $\{ 1,\ldots ,n\}$ and work with permutations. Thus, our enumerations have
the same expressive power as the permutations but we avoid ambiguity of composed permutations on $\{ 1,\ldots ,n\}$.
\smallskip

We introduce a special {\em payoff-array transformation},
which assigns to every game $m$ a real $\Upsilon\times N$-array, formally an element $x^{m}\in {\dv R}^{\Upsilon\times N}$,
that is, a function from the Cartesian product $\Upsilon\times N$ to the real line ${\dv R}$.
Specifically, we put
\begin{equation}
x^{m}(\pi ,i) := m\left(\bigcup_{k\leq\pi_{-1}(i)} \{\pi (k)\}\right) -  m\left(\bigcup_{k<\pi_{-1}(i)} \{\pi (k)\}\right),
\label{eq.payoff-array}
\end{equation}
for every $\pi\in\Upsilon$ and every $i\in N$.
For any $\pi\in\Upsilon$, the row-vector $x^{m}(\pi,\ast)\in {\dv R}^{N}$ is
nothing but what is named in game-theoretical literature the {\em marginal vector\/} of $m$
with respect to $\pi$, despite different notation; compare \eqref{eq.payoff-array}
with the definition from \cite{vVHN04}. Thus, the entry $x^{m}(\pi, i)$ can be interpreted as
the payoff to the player $i\in N$ provided that the distribution of the overall worth $m(N)$ is based
on the ordering of players given by the enumeration $\pi\in\Upsilon$.\footnote{This
interpretation comes from an implicit assumption that $A\subseteq B\subseteq N$ implies $m(A)\leq m(B)$.}
Therefore, the $\Upsilon\times N$-array given by \eqref{eq.payoff-array} is a kind of {\em payoff array\/}
in a~general sense, as discussed in \S\,14.5 of \cite{LR89}.

Clearly, the mapping $m\mapsto x^{m}$ is an invertible linear transformation: the linearity
follows directly from the formula \eqref{eq.payoff-array}, the invertibility from the fact that,
for any $\pi\in\Upsilon$, the row $x^{m}(\pi, \ast)$ of the array is in a~one-to-one correspondence with the restriction of $m$ to the maximal chain
${\cal C}_{\pi}$ of sets, defined by
\begin{equation}
{\cal C}_{\pi}:\qquad \emptyset\quad \{\pi (1)\}\quad \{ \pi (1),\pi (2)\}\quad \ldots \quad
\{ \pi(1),\pi (2),\ldots ,\pi (n)\}\equiv N\,.
\label{eq.max-chain}
\end{equation}
Note that we intentionally include the empty set $\emptyset$
into the (maximal) chain ${\cal C}_{\pi}$; this becomes convenient later.
More specifically, observe that one has
\begin{eqnarray*}
x^{m}(\pi , \pi(1)) &=& m(\{\pi (1)\})- m(\emptyset ) ~=~ m(\{\pi (1)\})\,,\\
x^{m}(\pi , \pi(l)) &=& m(\{\pi (1),\ldots ,\pi (l)\})
- m(\{\pi (1),\ldots ,\pi (l-1)\})\quad \mbox{for $2\leq l\leq n$}.
\end{eqnarray*}
Conversely, by inductive consideration one can easily observe
\begin{equation}
\forall ~\mbox{game $m\in {\dv R}^{\caP}$}\quad \forall\, \pi\in\Upsilon
\qquad S\in {\cal C}_{\pi} ~\Rightarrow~ \sum_{i\in S}\, x^{m}(\pi ,i)= m(S)\,.
\label{eq.chain}
\end{equation}

\begin{defin}[Weber set]\label{def.weber-set}\rm ~\\
Every game $m$ over $N$ is assigned the {\em Weber set\/}, defined as the convex hull of the
set of rows of the above-mentioned array $x^{m}$:
$$
\web (m) := \conv (\{ x^{m}(\tau ,\ast)\in {\dv R}^{N} \, :\ \tau\in\Upsilon\})\,.
$$
\end{defin}

A well-known fact is that the inclusion $\cor(m)\subseteq \web (m)$ holds for any game $m$; see Theorem~14 in \cite{web88}. We base our proof on the following fact, which is a corollary of
Theorem \ref{thm:super} in \ref{sec.apex-supermod}.

\begin{lem}\label{lem.1}\rm
A game $m$ over $N$ is supermodular iff the vertices of its Weber set $\web (m)$ give
a {\em min-representation\/} of $m$, more precisely, iff
\begin{equation}
\forall\, S\subseteq N\qquad
m(S) = \min_{\tau\in\Upsilon}\,\, \sum_{i\in S}\, x^{m}(\tau ,i)\,.
\label{eq.min-weber}
\end{equation}
Supposing this is the case (= $m$ is supermodular) one has $\web (m)=\cor (m)$.
\end{lem}

In \ref{sec.apex-supermod} we have also collected a number of other equivalent definitions of supermodularity of a~game.

\subsection{Obtaining the vertices of the core for a supermodular game}

Another notable fact is that, for a supermodular game $m$, none of the rows in the payoff array $x^{m}$ given by \eqref{eq.payoff-array}
is a non-trivial convex combination of others; in fact, they can only be repeated.
A basic observation to derive this fact is that, if $m\in\supermo$, then
\begin{equation}
\forall\, \tau\in\Upsilon ~~ \forall\, S\subseteq N\quad
\sum_{i\in S}\,  x^{m}(\tau ,i)-m(S)\geq 0\,,\qquad \mbox{which follows from Lemma \ref{lem.1}.}
\label{eq.non-neg}
\end{equation}

\begin{lem}\label{lem.aux}\rm
Given $m\in\supermo$, $\tau\in\Upsilon$ and $\Gamma\subseteq\Upsilon\setminus\{\tau\}$ such that
$$
x^{m}(\tau ,\ast) =\sum_{\pi\in\Gamma}\, \alpha_{\pi}\cdot x^{m}(\pi,\ast) \qquad
\mbox{for some $\alpha_{\pi}>0$ with~ $\sum_{\pi\in\Gamma} \alpha_{\pi}=1$,}
$$
then one has
$$
\forall\, \pi\in\Gamma\quad x^{m}(\pi,\ast) =x^{m}(\tau ,\ast)\,.
$$
\end{lem}

\begin{proof}
By \eqref{eq.chain}, one can write for any $S\in {\cal C}_{\tau}$:
\begin{eqnarray*}
0&=&\sum_{i\in S}\, x^{m}(\tau ,i)-m(S) = \sum_{i\in S} \sum_{\pi\in\Gamma}\, \alpha_{\pi}\cdot x^{m}(\pi,i) -m(S)\\ & =&
 \sum_{\pi\in\Gamma}\, \alpha_{\pi}\cdot \sum_{i\in S} x^{m}(\pi,i) -m(S)\\
&=&  \sum_{\pi\in\Gamma}\, \alpha_{\pi}\cdot \sum_{i\in S} x^{m}(\pi,i) -  \sum_{\pi\in\Gamma}\, \alpha_{\pi}\cdot m(S) =
 \sum_{\pi\in\Gamma}\, \alpha_{\pi}\cdot \underbrace{\left( \sum_{i\in S} x^{m}(\pi,i) - m(S) \right)}_{\geq 0}\,,
\end{eqnarray*}
where the inner expressions in braces are non-negative by \eqref{eq.non-neg}.
Therefore, since $\alpha_{\pi}>0$ for $\pi\in\Gamma$, they all must vanish. Thus, for any $\pi\in\Gamma$, one has
$$
\forall\, S\in {\cal C}_{\tau}\qquad \sum_{i\in S}\, x^{m}(\pi,i) = m(S)=\sum_{i\in S}\, x^{m}(\tau,i)\,.
$$
Hence, for any fixed $\pi\in\Gamma$, by inductive consideration, $x^{m}(\pi,i) =x^{m}(\tau ,i)$ for $i\in N$.
\end{proof}

This gives a simple procedure to get all vertices of the core for a supermodular
game, mentioned already in 1972 by Shapley; see Theorems 3 and 5 in \cite{sha72}.

\begin{coro}\label{cor.vert-core-supermod}\rm
Given $m\in\supermo$, one can obtain the set of (all) vertices of $\cor (m)$ by discarding
the repeated occurrences of rows in the payoff array \eqref{eq.payoff-array}.
\end{coro}

\begin{proof}
By Lemma \ref{lem.1}, $\web (m)=\cor (m)$ and, thus, by Definition \ref{def.weber-set},  the vertex set of $\cor (m)$ is a subset of the set of rows in
\eqref{eq.payoff-array}. By Lemma \ref{lem.aux}, after the removal of all repeated occurrences, none of the rows of the pruned array is a convex combination of others.
\end{proof}

\section{The core-based criterion for extremality}\label{sec.characterization}
The cone $\supmogame$ is a pointed polyhedral cone and, therefore, has finitely many extreme rays.
Our main result is a necessary and sufficient condition for non-zero $m\in\supmogame$
to generate an extreme ray of $\supmogame$.

\begin{remark}\rm\label{rem.exteme}
One can introduce extreme supermodular games as follows:
a game $m\in\supermo$ will be called {\em extreme\/} if $m=m^{\star}+g$,
where $m^{\star}\neq 0$ generates an extreme ray of $\supmogame$  and $g$ is a modular game.
Therefore, to test the extremality of $m$ one first applies \eqref{eq.standard}
to $m$ and then tests whether $m^{\star}$ generates an extreme ray of $\supmogame$.

Another note is that some other authors \cite{RW73,RW74}, \cite[\S\,4]{QdC08}, \cite[\S\,2]{kas00} consider the pointed cone of
non-negative supermodular games instead of $\supmogame$.
Nevertheless, this is only an~inessential modification since, by \eqref{eq.standard}, one can write that cone as the Minkowski sum of $\supmogame$
and the cone spanned by $\{ m^{\uparrow i}:\, i\in N\}$, while such decomposition of any game in that cone is unique.
In particular, the only additional extreme rays of their cone besides the
extreme rays of $\supmogame$ are the rays generated by $m^{\uparrow i}$, $i\in N$.
Further minor technical difference is that some of these authors  regard the zero game as extreme, e.g.\ \cite{kas00}.
\end{remark}

\subsection{Formulation of the main result}\label{ssec.main-formul}
Our criterion is, in fact, a geometrical condition on the
set of vertices  of the core of $m$, denoted below occasionally by ${\cal X}:=\ext (\cor (m))$.\footnote{The symbol $\ext (P)$
is used to denote the set of vertices (= extreme points) of a polytope $P$ in ${\dv R}^{N}$.}
Technically, our criterion is formulated in terms of {\em any real array $x\in {\dv R}^{\Gamma\times N}$} of the form
\begin{equation}
x=[x(\tau ,{i})]_{\tau\in\Gamma,\, i\in N} \quad
\mbox{such that}~~ {\cal X}=\ext (\cor (m)) = \{\, [x(\tau ,{i})]_{i\in N}\in {\dv R}^{N}\, :~ \tau\in\Gamma\,\}\,,
\label{eq.gamma-array}
\end{equation}
that is, the set of {\em distinct rows of\/ $x$} coincides with the set of (all) {\em vertices of\/ $\cor (m)$}. The role of the indexing set $\Gamma$
in \eqref{eq.gamma-array} is auxiliary; neither the order of the rows nor the order of the columns matters.
Also, repeating the rows has no influence, as shown below. Nevertheless, the maximally pruned arrays without repeated rows are preferred.

Given a standardized supermodular game $m\in\supmogame$, the elements of the respective payoff array \eqref{eq.payoff-array} are
non-negative. This is because every standardized supermodular function is
non-decreasing with respect to inclusion. Thus, it follows from Corollary \ref{cor.vert-core-supermod},
that the maximally pruned array \eqref{eq.gamma-array} is unique up to re-ordering of rows and its entries are non-negative.
Moreover, by Lemmas \ref{lem.1} and \ref{lem.aux}, the assumed array \eqref{eq.gamma-array} uniquely determines the game $m$
through \eqref{eq.min-weber} with $\Upsilon$ replaced by $\Gamma$. More specifically, one has
\begin{equation}
\forall\, S\subseteq N\qquad
m(S) = \min_{\tau\in\Gamma}\,\, \sum_{i\in S}\, x(\tau ,i)\,.
\label{eq.min-gamma-array}
\end{equation}

\begin{defin}[null-set, tightness set class]\label{def.tight-sets}\rm ~\\
We introduce, for any row $\tau\in\Gamma$ of the considered array \eqref{eq.gamma-array}
\begin{eqnarray*}
N_{\tau} & := & \{ i\in N :~ x(\tau ,i)=0\}, ~\mbox{the {\em null-set\/} of the row-vector $x(\tau ,\ast)\in {\dv R}^{N}$},\\
{\cal S}^{m}_{\tau} & := & \{ S\subseteq N :~  m(S)=\sum_{i\in S}  x(\tau ,i)\},\\
&& \mbox{the {\em class of the sets\/} at which the row-vector $x(\tau ,\ast)$ is {\em tight\/} with $m$}\,.
\end{eqnarray*}
\end{defin}

A notable fact is that the {\em tightness sets} ${\cal S}^{m}_{\tau}$, for $\tau\in\Gamma$,
can equivalently be introduced solely in terms of the array \eqref{eq.gamma-array}.
Indeed, because of \eqref{eq.min-gamma-array}, one has, for any $\tau\in\Gamma$,
$$
{\cal S}^{m}_{\tau} = {\cal S}^{x}_{\tau} :=~ \{ S\subseteq N :~  \forall\, \pi\in\Gamma\quad
\sum_{i\in S}  x(\tau ,i)\leq \sum_{i\in S}  x(\pi ,i)\,\}\,.
$$
Thus, one can write ${\cal S}_{\tau}^{x}$ instead of
${\cal S}^{m}_{\tau}$. When $x$ is fixed and there is no danger of confusion, we
omit the upper index and write just ${\cal S}_{\tau}$.
Now, we introduce a system of linear constraints for real arrays
$y\in {\dv R}^{\Gamma\times N}$:

\bigskip
\fbox
{
\begin{minipage}{0.9\textwidth}
\begin{itemize}
\item[(a)] $\forall\, \tau\in\Gamma \qquad \text{if $i\in N_{\tau}$, then $y(\tau ,i) =0$,}$
\item[(b)] $\forall S\subseteq N\enskip\forall\, \tau ,\pi\in\Gamma\text{ such that $S\in {\cal S}_{\tau}\cap {\cal S}_{\pi}$}
~\quad \sum_{i\in S} y(\tau ,i) =\sum_{i\in S} y(\pi ,i)$.
\end{itemize}
\end{minipage}
}

\bigskip
\noindent
It is not difficult to observe that the starting array $x\in {\dv R}^{\Gamma\times N}$ from \eqref{eq.gamma-array}
satisfies these linear constraints. Informally, the characterization is that the structural information
given by sets $N_{\tau}$ and ${\cal S}_{\tau}$, for all $\tau\in\Gamma$, already determines the array up to a real multiple.

\begin{thm}\label{thm.1}
Let $m\in\supmogame$ be a non-zero standardized supermodular game. Consider a real array $x\in {\dv R}^{\Gamma\times N}$
of the form \eqref{eq.gamma-array}. Then $m$ generates an extreme ray of
$\supmogame$ iff every real solution $y\in {\dv R}^{\Gamma\times N}$ to (a)-(b) is a multiple of
$x$, i.e.\
$$
\exists\, \alpha\in {\dv R} \ : \qquad
y(\tau , i) =\alpha\cdot x(\tau, i)\quad\mbox{for any $\tau\in\Gamma$ and $i\in N$.}
$$
\end{thm}

The proof of Theorem \ref{thm.1} is given in \S\,\ref{sec.proof}.

\subsection{Examples}

First, we illustrate the use of Theorem \ref{thm.1} by simple examples of supermodular games.

\begin{example}\label{exa.positive}
Put $N=\{ a,b,c\}$ and define a game $m$ over $N$ by
$$
m(S):= |S|-1\quad \text{for every non-empty $S\subseteq N$.}
$$
Then $m\in \supmogame$ and the core of $m$ is a translated reflection of the standard simplex:
$$
\cor (m)= \conv (\{\, [0,1,1],[1,0,1],[1,1,0]\,\}).
$$
The respective array $x$ satisfying \eqref{eq.gamma-array}, without repeated rows, has the form
$$
 x =\bbordermatrix{
  & a & b & c \cr
 \pi & 0 & 1 &ÃÂ 1 \cr
 \sigma & 1 & 0 &ÃÂ 1 \cr
 \eta & 1 & 1 &ÃÂ 0 \cr
 },\quad \text{where $\Gamma =\{\pi,\sigma,\eta\}$.}
$$
By Corollary \ref{cor.vert-core-supermod}, it can alternatively be obtained by discarding the repeating occurrences
of rows in the respective payoff array \eqref{eq.payoff-array} with six rows.
The array $x$ yields the following null-sets and classes of tightness sets:
\begin{eqnarray*}
 &N_{\pi}=\{a\} &~~ {\cal S}_{\pi} = \{\, \emptyset,\{a\},\{a,b\},\{a,c\},N\,\},\\
 &N_{\sigma}=\{b\} &~~ {\cal S}_{\sigma}=\{\, \emptyset,\{b\},\{a,b\},\{b,c\},N\,\},\\
 &N_{\eta}=\{c\} &~~ {\cal S}_{\eta}=\{\, \emptyset,\{c\},\{a,c\},\{b,c\},N\,\}.
\end{eqnarray*}
Assume that $y\in{\dv R}^{\Gamma\times N}$ satisfies the conditions (a)-(b)
from \S\,\ref{ssec.main-formul}. Then (a) says that the array $y$ has necessarily the following form:
$$
 y=\bbordermatrix{
  & a & b & c \cr
 \pi & 0 & y(\pi ,b) &Â y(\pi ,c) \cr
 \sigma & y(\sigma ,a) & 0 &Â y(\sigma ,c) \cr
 \eta & y(\eta ,a) & y(\eta ,b) &Â 0 \cr
 }.
$$
The condition (b) then requires
\begin{eqnarray*}
 y(\pi,b) \stackrel{\mbox{\scriptsize (a)}}{=} y(\pi,a)+y(\pi,b)
 & \stackrel{\mbox{\scriptsize (b)}}{=}&
 y(\sigma,a)+y(\sigma,b) \stackrel{\mbox{\scriptsize (a)}}{=} y(\sigma,a)
 \quad \mbox{for $\{a,b\}\in {\cal S}_{\pi}\cap{\cal S}_{\sigma}$,}\\
 y(\pi,c) \stackrel{\mbox{\scriptsize (a)}}{=} y(\pi,a)+y(\pi,c)
 & \stackrel{\mbox{\scriptsize (b)}}{=}&
 y(\eta,a)+y(\eta,c) \stackrel{\mbox{\scriptsize (a)}}{=} y(\eta ,a)
 \quad \mbox{for $\{a,c\}\in {\cal S}_{\pi}\cap{\cal S}_{\eta}$,}\\
 y(\sigma ,c) \stackrel{\mbox{\scriptsize (a)}}{=} y(\sigma,b)+y(\sigma,c)
 & \stackrel{\mbox{\scriptsize (b)}}{=}&
 y(\eta,b)+y(\eta,c) \stackrel{\mbox{\scriptsize (a)}}{=} y(\eta ,b)
 \quad \mbox{for $\{b,c\}\in {\cal S}_{\sigma}\cap{\cal S}_{\eta}$.}
\end{eqnarray*}
Thus, we necessarily have
$$
 y=\bbordermatrix{
  & a & b & c \cr
 \pi & 0 & y_{1} &Â y_{2} \cr
 \sigma & y_{1} & 0 &Â y_{3} \cr
 \eta & y_{2} & y_{3} &Â 0 \cr
 }.
$$
Since $N\in {\cal S}_{\pi}\cap {\cal S}_{\sigma}\cap {\cal S}_{\eta}$ the condition (b)
also gives $\sum_{i\in N}y(\pi,i)=\sum_{i\in N}y(\sigma,i)=\sum_{i\in N}y(\eta,i)$,
that is, $y_{1}+y_{2}=y_{1}+y_{3}=y_{2}+y_{3}$ implying $y_{1}=y_{2}=y_{3}$.
We can conclude that the linear system (a)-(b) has all the solutions
in the form $y=\alpha\cdot x$, where $\alpha\in {\dv R}$. Therefore, $m$ is extreme in $\supmogame$
by Theorem \ref{thm.1}.
\end{example}

The second example shows how non-extremality of a supermodular game can be verified easily.

\begin{example}\label{exa.negative}
Assume that $N=\{a,b,c\}$. Let $\gamma$ be an enumeration of $N$ such that $\gamma(1)=a$,
$\gamma(2)=b$, and $\gamma(3)=c$. Put
$$
t(S)=\left(\sum\limits_{i\in S}\gamma_{-1}(i)\right)^{2} \quad \text{for every $S\subseteq N$.}
$$
Then $t$ is a supermodular game, namely the so-called {\em convex measure game},
discussed already by Shapley \cite[\S\,2.2]{sha72}; note that
we recall a particular extremality criterion for convex measure games in \S\,\ref{ssec.conv-measure-game},.

It is well-known that $t$
lies in the relative interior of the supermodular cone $\supermo$; thus, it is not extreme.
Moreover, $t\notin \supmogame$ since $t(S)\neq 0$ for $S\subseteq N$ with $|S|=1$.
Let's apply the standardization formula \eqref{eq.standard} and put
$$
t^{\star}(S)=t(S)-\sum_{i\in S}\, t(\{i\}) \quad \text{for every $S\subseteq N$.}
$$
In fact, $t^{\star}=22\cdot\delta_{N} +4\cdot\delta_{\{ a,b\}} +6\cdot\delta_{\{ a,c\}} +12\cdot\delta_{\{ b,c\}}$
is in the relative interior of $\supmogame$. The core of $t^{\star}$ is a~hexagon whose vertices are
detailed in the rows of the following array:
$$
x=\bbordermatrix{
  & a & b & c \cr
\mu & 0 & 4 & 18 \cr
\nu & 4Â & 0 & 18\cr
\pi & 0 & 16 &Â 6 \cr
\rho & 6 & 16Â & 0\cr
\sigma & 10 & 0 &Â 12 \cr
\eta & 10 & 12 &Â 0 \cr
},
\quad \text{where $\Gamma=\{\mu,\nu,\pi,\rho,\sigma,\eta\}$.}
$$
The null-sets and the tightness sets are as follows:
\begin{eqnarray*}
 &N_{\mu}=\{a\} &~~ {\cal S}_{\mu} = \{\,\emptyset,\{a\},\{a,b\},N\,\},\\
 &N_{\nu}=\{b\} &~~ {\cal S}_{\nu}=\{\,\emptyset,\{b\},\{a,b\},N\,\},\\
 &N_{\pi}=\{a\} &~~ {\cal S}_{\pi}=\{\,\emptyset,\{a\},\{a,c\},N\,\},\\
 &N_{\rho}=\{c\} &~~ {\cal S}_{\rho} = \{\,\emptyset,\{c\},\{a,c\},N\,\},\\
 &N_{\sigma}=\{b\} &~~ {\cal S}_{\sigma}=\{\,\emptyset,\{b\},\{b,c\},N\,\},\\
 &N_{\eta}=\{c\} &~~ {\cal S}_{\eta}=\{\,\emptyset,\{c\},\{b,c\},N\,\}.
\end{eqnarray*}
Observe that each tightness set correspond to a maximal chain in $\caP$.
It is easy to verify that the array $y\in{\dv R}^{\Gamma\times N}$ given by
$$
 y=\bbordermatrix{
  & a & b & c \cr
 \mu & 0 & 0& 22 \cr
 \nu & 0Â & 0 & 22\cr
 \pi & 0 & 16 &Â 6 \cr
 \rho & 6 & 16Â & 0\cr
 \sigma & 10 & 0 &Â 12 \cr
 \eta & 10 & 12 &Â 0 \cr
 },
$$
meets the conditions (a)-(b) from \S\,\ref{ssec.main-formul} and, despite, $y$ is not a real multiple of $x$.
Thus, $m$ is not an extreme game by Theorem \ref{thm.1}.
\end{example}

The next example is slightly aside the topic because it is an exact game which is {\em not\/} supermodular.
It only illustrates that the condition from  Theorem \ref{thm.1} can be considered outside the supermodular framework,
although our result does not apply in this particular case. We conjecture that our condition from Theorem \ref{thm.1}
is necessary for an exact game to generate an extreme ray of the cone of standardized exact games.

\begin{example}\label{exa.Doker}
Assume $N=\{ a,b,c,d\}$ and consider the game $m$ over $N$ given by
$$
m=4\cdot\delta_{N} + 3\cdot\delta_{\{a,b,c\}} + 2\cdot\delta_{\{a,b,d\}}
+ 2\cdot\delta_{\{a,c,d\}} + 2\cdot\delta_{\{b,c,d\}} +  2\cdot\delta_{\{a,b\}}
+ 2\cdot\delta_{\{a,c\}} + 2\cdot\delta_{\{b,c\}}.
$$
The game $m$ is not supermodular as $m(\{a,c\})+m(\{b,c\})=4>3=m(\{a,b,c\})+m(\{c\})$.
Its core belongs to the plane $x_{d}=4-x_{a}-x_{b}-x_{c}$ and has four facet-defining inequalities:
$$
x_{a}+x_{b}+x_{c}\leq 4,\quad 2\leq x_{a}+x_{b},\quad 2\leq x_{a}+x_{c},\quad 2\leq x_{b}+x_{c}.
$$
One can easily check that the core $\cor (m)$ has four vertices $[x_{a},x_{b},x_{c},x_{d}]$, namely
$[\,1, 1, 1, 1\,]$, $[\,2, 2, 0, 0\,]$, $[\,2, 0, 2, 0\,]$, $[\,0, 2, 2, 0\,]$.
This allows one to verify that every inequality in \eqref{eq.def.core} is tight
for some $v\in {\cal X}=\ext (\cor (m))$. In other words, the game $m$ is exact, which means
$$
m(S) =\min_{v\in\cor (m)}\, \sum_{i\in S}\, v_{i}= \min_{v\in {\cal X}}\, \sum_{i\in S}\, v_{i}\qquad
\mbox{for any $S\subseteq N$.}
$$
Our computation of the extreme rays of the (polyhedral) cone of exact
standardized games over four variables confirmed that $m$ is an extreme exact game over $N$.
Let us arrange the vertices of $\cor (m)$ into a $\Gamma\times N$-array with $\Gamma =\{\pi,\rho,\eta,\sigma\}$:
\begin{equation}
\bbordermatrix{
& a & b & c & d \cr
\pi & 1 & 1 & 1 & 1 \cr
\rho & 2 & 2 & 0 & 0 \cr
\eta & 2 & 0 & 2 & 0 \cr
\sigma & 0 & 2 & 2 & 0 \cr
}.
\label{eq.anti-Doker-example}
\end{equation}
Assume that $y\in{\dv R}^{\Gamma\times N}$ satisfies the conditions (a)-(b)
from \S\,\ref{ssec.main-formul}. The following are the sets $N_{\tau}$ and ${\cal S}_{\tau}$ for
$\tau\in\Gamma$:
\begin{eqnarray*}
&N_{\pi}=\emptyset &~~ {\cal S}_{\pi} = \{\, \emptyset , \{a,b\}, \{a,c\}, \{b,c\}, \{a,b,c\}, N\,\},\\
&N_{\rho}=\{c,d\} &~~ {\cal S}_{\rho} = \{\, \emptyset , \{ c\}, \{ d\}, \{a,c\}, \{b,c\}, \{c,d\}, \{a,c,d\}, \{b,c,d\}, N\,\},\\
&N_{\eta}=\{b,d\} &~~ {\cal S}_{\eta} = \{\, \emptyset ,  \{ b\}, \{ d\}, \{a,b\}, \{b,c\},  \{b,d\}, \{a,b,d\}, \{b,c,d\}, N\,\},\\
&N_{\sigma}=\{a,d\} &~~ {\cal S}_{\sigma} =  \{\, \emptyset ,  \{ a\}, \{ d\}, \{a,b\}, \{a,c\},  \{a,d\}, \{a,b,d\}, \{a,c,d\}, N\,\}.
\end{eqnarray*}
The condition (a) implies that $y$ has the form
$$
\left[
\begin{matrix}
 & y(\pi ,a) & y(\pi, b) & y(\pi,c) & y(\pi ,d)  \\
~& y(\rho ,a)  & y(\rho, b)  & 0 & 0 \\
~& y(\eta ,a) & 0 & y(\eta, c) & 0 \\
~& 0 & y(\sigma, b) & y(\sigma,c) & 0
\end{matrix}
~~~\right].
$$
Now, the condition (b) implies
\begin{eqnarray*}
\{a,b\}\in {\cal S}_{\pi}\cap{\cal S}_{\eta}\cap {\cal S}_{\sigma} &\Rightarrow&
y(\pi,a)+y(\pi, b) = y(\eta, a) =y (\sigma, b) ~=:~ U,\\
\{a,c\}\in {\cal S}_{\pi}\cap{\cal S}_{\rho}\cap {\cal S}_{\sigma} &\Rightarrow&
y(\pi,a)+y(\pi, c) = y(\rho, a) =y (\sigma, c) ~=:~ V,\\
\{b,c\}\in {\cal S}_{\pi}\cap{\cal S}_{\rho}\cap {\cal S}_{\eta} &\Rightarrow&
y(\pi,b)+y(\pi, c) = y(\rho, b) =y (\eta, c) ~=:~ W.
\end{eqnarray*}
Because $N\in {\cal S}_{\pi}\cap{\cal S}_{\rho}\cap {\cal S}_{\eta}\cap {\cal S}_{\sigma}$
the condition (b), moreover, gives
\begin{eqnarray*}
\lefteqn{y(\pi,a)+y(\pi, b)+y(\pi,c)+y(\pi, d)}\\
&=& \hspace*{-2mm}\underbrace{y(\rho, a)+y(\rho, b)}_{V+W}=\underbrace{y (\eta, a)+y (\eta, c)}_{U+W}=
\underbrace{y (\sigma, b)+y (\sigma, c)}_{U+V}= V+W = U+ W = U+ V,
\end{eqnarray*}
implying $U=V=W$. Then again $y(\pi,a)+y(\pi, b)=U=V=y(\pi,a)+y(\pi, c)$ implies
$y(\pi, b)=y(\pi, c)$ and analogously $y(\pi,a)+y(\pi, b)=U=W=y(\pi,b)+y(\pi, c)$ implies
$y(\pi, a)=y(\pi, c)$. Hence,  $y(\pi, a)=y(\pi, b)=y(\pi, c)=\frac{U}{2}$ and
the above equalities give
$$
\frac{3}{2}\cdot U+y(\pi, d)= y(\pi,a)+y(\pi, b)+y(\pi,c)+y(\pi, d)
\stackrel{\mbox{\scriptsize (b)}}{=} y (\sigma, b)+y (\sigma, c)=
U+V=2\cdot U
$$
implying $y(\pi, d)=\frac{U}{2}$. In particular, any solution $y$ to (a)-(b) is the
$\frac{U}{2}$-multiple of the original array \eqref{eq.anti-Doker-example}.
The condition from \S\,\ref{ssec.main-formul} is, therefore, fulfilled.
\end{example}

Nonetheless, the condition from Theorem \ref{thm.1} is not sufficient for an exact game
to generate an extreme ray of the cone of standardized exact games as the following example shows.

\begin{example}\label{exa.Vamosi}
Put $N=\{ a,b,c,d\}$ and consider the following special game over $N$:
$$
m_{\dag} = 4\cdot\delta_{N} + 2\cdot\delta_{\{a,b,c\}} + 2\cdot\delta_{\{a,b,d\}}
+ 2\cdot\delta_{\{a,c,d\}} + 2\cdot\delta_{\{b,c,d\}}
+ \delta_{\{a,b\}} + \delta_{\{a,c\}} + \delta_{\{a,d\}} + \delta_{\{b,c\}} + \delta_{\{b,d\}}.
$$
It is easy to see that $m_{\dag}\in\supmogame$. Actually, one can verify by
Theorem \ref{thm.1} that $m_{\dag}$ generates an extreme ray of $\supmogame$.
In this case, $\cor (m_{\dag})$ has 13 vertices  and 13 facets as well. Specifically, the vertices are
detailed in the following array:
$$
x_{\dag}=\bbordermatrix{
~& a & b & c & d\cr
 & 2 & 2 & 0 & 0\cr
 & 2Â & 1 & 1 & 0\cr
 & 2 & 1 &Â 0 & 1\cr
 & 2 & 0Â & 1 & 1\cr
 & 1 & 2 &Â 1 & 0\cr
 & 1 & 2 &Â 0 & 1\cr
 & 1 & 1 & 2 & 0\cr
 & 1Â & 1 & 0 & 2\cr
 & 1 & 0 &Â 2 & 1\cr
 & 1 & 0Â & 1 & 2\cr
 & 0 & 2 &Â 1 & 1\cr
 & 0 & 1 &Â 2 & 1\cr
 & 0 & 1 &Â 1 & 2\cr
}.
$$
It is tedious but straightforward to verify directly that every solution to (a)-(b) in this case is a multiple of $x_{\dag}$.
Thus, by Theorem \ref{thm.1}, $m_{\dag}$ is an extreme supermodular game.
Nevertheless, $m_{\dag}$ is {\em not} extreme in the cone of standardized exact games. This follows
from the relation $m_{\dag}=m_{0}+m_{1}$ where
\begin{eqnarray*}
m_{0} &=& 2\cdot\delta_{N} + \delta_{\{a,b,c\}} + \delta_{\{a,b,d\}} + \delta_{\{a,c,d\}} + \delta_{\{b,c,d\}}
 + \delta_{\{a,c\}} + \delta_{\{b,c\}} + \delta_{\{b,d\}}\,,\\
m_{1} &=& 2\cdot\delta_{N} + \delta_{\{a,b,c\}} + \delta_{\{a,b,d\}} + \delta_{\{a,c,d\}} + \delta_{\{b,c,d\}}
 + \delta_{\{a,b\}} + \delta_{\{a,d\}}\,.
\end{eqnarray*}
The point is that both $m_{0}$ and $m_{1}$ are exact games, which are not supermodular.
Their cores have three and four vertices, respectively, shown in the following arrays:
$$
x_{0}=\bbordermatrix{
~& a & b & c & d\cr
 & 1 & 1 & 0 & 0\cr
 & 0Â & 1 & 1 & 0\cr
 & 0 & 0 &Â 1 & 1\cr
},
\qquad\quad
x_{1}=\bbordermatrix{
~& a & b & c & d\cr
 & 1 & 1 & 0 & 0\cr
 & 1Â & 0 & 1 & 0\cr
 & 1 & 0 &Â 0 & 1\cr
 & 0 & 1Â & 0 & 1\cr
}.
$$
Actually, our computation of the extreme rays of the cone of exact standardized games confirmed that
both $m_{0}$ and $m_{1}$ generate extreme rays of that cone. We leave it to the reader as an easy
exercise to verify that they both satisfy the condition from Theorem \ref{thm.1}. As concerns their convex combinations
$m_{\lambda}:=(1-\lambda)\cdot m_{0}+ \lambda\cdot m_{1}$, $\lambda\in [0,1]$, the games $m_{\lambda}$
for $\lambda\in (0,1)\setminus\{\frac{1}{2}\}$ have cores with sixteen vertices and none of them satisfies the condition from Theorem \ref{thm.1}.
\end{example}

\subsection{Interpretation of Theorem \ref{thm.1}}\label{ssec.interpret}
What follows is a minor modification of the definition given by Kuipers et al.\/ \cite[\S\,2]{KVV10}.
The below defined concept has been introduced for general balanced games;
however, we believe it is particularly useful and important in the context of supermodular games.

\begin{defin}[core structure]\label{def.core-structure}\rm ~\\
Let $m$ be a balanced game over $N$.
By the {\em core structure\/} of $m$ we will understand a mapping
which assigns to every vertex $v$ of the core $\cor (m)$ the class of the respective tightness sets
(see Definition \ref{def.tight-sets}):
$$
v=[v_{i}]_{i\in N}\in {\cal X}=\ext (\cor (m)) ~\longmapsto~ {\cal S}^{m}_{v} =
\left\{ S\subseteq N :~  m(S)=\sum_{i\in S}  v_{i}\right\}.
$$
\end{defin}

Note that ``indexing" the classes of tightness sets by vertices of $\cor (m)$
only plays auxiliary role. One can alternatively and equivalently introduce
the core structure as a collection of subsets of the power set $\caP$, namely as
$$
\{\, {\cal S}^{m}_{v}\,:~ v\in\ext (\cor (m))\,\,\}\qquad
\mbox{which is basically the definition from \cite[\S\,2]{KVV10}.}
$$
Such an un-indexed collection of subsets of $\caP$ is already a fully
{\em combinatorial} concept, without any obvious geometric meaning.
The aim of our definition is to emphasize the expected geometric interpretation
of such combinatorial concept: the subsets of $\caP$  in the collection
should correspond to the vertices of the core.

A relevant observation is that in case of a supermodular game $m$,
the combinatorial core structure is non-empty finite {\em collection\/} of {\em sub-lattices} of the lattice $(\caP,\subseteq )$. Indeed, Theorem \ref{thm:super}\eqref{i:tesnoset} from \ref{sec.apex-supermod} says that
every ${\cal S}^{m}_{v}$ for $v\in\cor (m)$ is
closed under intersection and union and one also has
$\emptyset,N\in {\cal S}^{m}_{v}$.
Moreover, different vertices of $\cor (m)$ give rise to incomparable
classes of tightness sets. This is because, for each pair of distinct
vertices, a facet of $\cor (m)$ exists containing just one of the vertices and this facet corresponds to a tightness set.
Therefore, for a supermodular
game $m$, the combinatorial view is always compatible with the ``geometric"
interpretation from Definition \ref{def.core-structure}.
\medskip

The standardization procedure \eqref{eq.standard} basically does not change
the core structure. The point is that, in our frame of {\em standardized supermodular\/} games,
the core structure of $m$ already fully determines the system of linear constraints (a)-(b) from \S\,\ref{ssec.main-formul}.
Indeed, assume without loss of generality that $\Gamma=\ext (\cor (m))$ in \eqref{eq.gamma-array} and observe
that $i\in N_{\tau}$ iff $\{i\}\in {\cal S}^{m}_{\tau}$, for $i\in N$ and $\tau\in\Gamma$.
Theorem \ref{thm.1} and the invertibility of the transformation
from \S\,\ref{ssec.payoff-trans} then imply that, provided $m$ generates an extreme ray of $\supmogame$, every $0\neq m^{\prime}\in\supmogame$
sharing the core structure with $m$ is necessarily a positive multiple of $m$, and, therefore,
$\cor (m^{\prime})$ is a dilation of $\cor (m)$. In other words:
\begin{quote}
if $m$ is extreme, then the {\em combinatorial core structure\/} of $m$\\  uniquely determines the
{\em geometric form\/} of the  core.
\end{quote}
We are convinced that this provides a simple and clear interpretation of our result.
Solving the linear equation system (a)-(b) from \S\,\ref{ssec.main-formul} then
allows one to verify/disprove extremality of $m$. The indeterminates in our system (a)-(b) are the pairs $(\tau ,i)$,
where $\tau$ corresponds to a {\em vertex} of the core and $i$ to a {\em variable}. It looks like that
our system substantially differs from former approaches just in this aspect.

\begin{remark}\rm\label{rem.Kuipers}
Kuipers  et al.\/ in their 2010 paper \cite{KVV10} introduced a further relevant concept. Specifically, they name
a game $g$ a {\em limit game} for a balanced game $m$ if, for every extreme point $v\in\ext(\cor(m))$,
an extreme point $w\in\ext (\cor(g))$ exists such that ${\cal S}^{m}_{v}\subseteq {\cal S}^{g}_{w}$, that is,
if $g$ has a coarser core structure than $m$. Also, they consider balanced games to be equivalent if
they have the same core structure.
To illustrate these concepts they show that the class of limit games for a {\em strictly supermodular game} $m$, that is,
$m$ satisfying
$$
m(A\cup B)+m(A\cap B) > m(A)+m(B)\quad \mbox{whenever $A,B\subseteq N$ with $A\setminus B\neq\emptyset\neq B\setminus A$,}
$$
is  just the class of supermodular games. This implies that the equivalent games to such a game $m$ are just the other strictly supermodular games.
Note that the set of all strictly supermodular games coincides with the relative interior of $\supermo$.
The main result of \cite{KVV10} characterizes the set of limit games $g$ for a balanced game $m$ in terms
of linear inequality constraints, which are, also uniquely determined by the core structure of $m$.
In contrast to our system of linear constraints (a)-(b) from \S\,\ref{ssec.main-formul}, these are constraints
on the game values $g(S)$, $S\subseteq N$, and obtaining those inequalities from the core structure is not straightforward.
\end{remark}

\section{Proof of the main result}\label{sec.proof}

We first prove Theorem \ref{thm.1} in a canonical special case when $x\in {\dv R}^{\Gamma\times N}$ is the
payoff array $x^{m}$ given by the formula \eqref{eq.payoff-array}. The following observation on the respective tightness sets
(see Definition \ref{def.tight-sets}) follows from \eqref{eq.chain} and is used repeatedly in the proof below:
\begin{equation}
\forall\,\tau\in\Upsilon \qquad
{\cal C}_{\tau}\subseteq {\cal S}^{m}_{\tau}\,.
\label{eq.chain-2}
\end{equation}

\begin{lem}\rm\label{lem.2}
Assuming $0\neq m\in\supmogame$, let $x=x^{m}$ be the array given by \eqref{eq.payoff-array}.\footnote{The
specialty of this array is that the row-index set $\Gamma$ is the set $\Upsilon$ of all enumerations for $N$.}
Then $m$ is extreme iff every real solution $y\in {\dv R}^{\Upsilon\times N}$ to (a)-(b)
is a multiple of $x^{m}$, that is,
$$
\exists\, \alpha\in {\dv R} \ : \qquad
y(\tau , i) =\alpha\cdot x^{m}(\tau, i)\quad\mbox{for any $\tau\in\Upsilon$ and $i\in N$.}
$$
\end{lem}

\begin{proof}
We show that the negation of the condition above, namely, the condition
\begin{equation}
\exists ~ \mbox{a solution $y\in {\dv R}^{\Upsilon\times N}$ of (a)-(b)} \ :
\qquad y\not\in \lin (x^{m})\,,
\label{eq.negative-conje}
\end{equation}
where $\lin (\ast)$ denotes the linear hull (in the respective space), is equivalent to the condition
\begin{equation}
\exists ~ \mbox{non-zero $r,s\in\supmogame$} \ :\qquad \lin (r)\neq\lin (s) \quad\mbox{and}\quad
 m=\frac{1}{2}\cdot r+ \frac{1}{2}\cdot s\,,
\label{eq.non-extreme}
\end{equation}
which is one of possible formulations of non-extremality of $m$ in the cone $\supmogame$.

To show \eqref{eq.non-extreme}$\Rightarrow$\eqref{eq.negative-conje} realize that the mapping
$m\mapsto x^{m}$ is an invertible linear transformation, which implies that $x^{r}$ and $x^{s}$ are both non-zero, $\lin (x^{r})\neq\lin (x^{s})$ and
\begin{equation}
x^{m}=\frac{1}{2}\cdot x^{r}+ \frac{1}{2}\cdot x^{s}\,.
\label{eq.decomp}
\end{equation}
We show that both $x^{r}$ and $x^{s}$ solve (a)-(b); since one of them is outside
$\lin (x^{m})$, it gives \eqref{eq.negative-conje}. One can derive that conclusion from the
fact that $x^{m}$ satisfies (a)-(b) using \eqref{eq.decomp}. To show (a) realize that,
for $\tau\in\Upsilon$ and $i\in N_{\tau}$ one has
$$
0 = x^{m}(\tau ,i) = \frac{1}{2}\cdot \underbrace{x^{r}(\tau ,i)}_{\geq 0}+ \frac{1}{2}\cdot \underbrace{x^{s}(\tau ,i)}_{\geq 0}\,,
$$
where the both terms on the right-hand side are non-negative. Indeed, realize we know $r,s\in\supmogame$, and,
therefore, their payoff-arrays are non-negative. Therefore, they must vanish: $x^{r}(\tau ,i)=0=x^{s}(\tau ,i)$.

As concerns (b), for $S\subseteq N$ and $\tau ,\pi\in\Upsilon$ such that $S\in {\cal S}_{\tau}\cap {\cal S}_{\pi}$
we first particularly write for $\tau$: because $S\in {\cal S}_{\tau}\equiv {\cal S}^{m}_{\tau}$ one has by \eqref{eq.decomp}
$$
0 = \sum_{i\in S}  x^{m}(\tau ,i)- m(S)=
\frac{1}{2}\cdot \underbrace{\left(\sum_{i\in S} x^{r}(\tau ,i)- r(S)\right)}_{\geq 0}+ \frac{1}{2}\cdot \underbrace{\left(\sum_{i\in S} x^{s}(\tau ,i)- s(S)\right)}_{\geq 0}\,,
$$
where the terms on the right-hand side must be non-negative by
\eqref{eq.non-neg}; realize $r,s\in\supmogame$. This gives both
$\sum_{i\in S} x^{r}(\tau ,i)=r(S)$ and $\sum_{i\in S} x^{s}(\tau ,i)=s(S)$.
The second step is to repeat the same consideration for $\pi$ in place of $\tau$ and derive both
$\sum_{i\in S} x^{r}(\pi ,i)=r(S)$ and $\sum_{i\in S} x^{s}(\pi ,i)=s(S)$.
Hence, $\sum_{i\in S} x^{r}(\tau ,i)=r(S)=\sum_{i\in S} x^{r}(\pi ,i)$
and analogously $\sum_{i\in S} x^{s}(\tau ,i)=s(S)=\sum_{i\in S} x^{s}(\pi ,i)$.
Thus, $x^{r}$ and $x^{s}$ both satisfy (b), which completes the proof of
\eqref{eq.non-extreme}$\Rightarrow$\eqref{eq.negative-conje}.
\smallskip

To show \eqref{eq.negative-conje}$\Rightarrow$\eqref{eq.non-extreme}, choose and fix
$y\in {\dv R}^{\Upsilon\times N}$ mentioned in \eqref{eq.negative-conje}.
The first step is to show that a game $t$ over $N$ exists such that $y=x^{t}$, that is,
$y$ is the range of our payoff-array transformation.
Recall from \S\,\ref{ssec.payoff-trans} that, provided $y=x^{t}$, for every $\pi\in\Upsilon$,
the respective row $y(\pi ,\ast)$ of the array $y\in {\dv R}^{\Upsilon\times N}$ determines
(and is determined by) the values of $t$ on the (maximal) chain ${\cal C}_{\pi}$ by the relation
\eqref{eq.chain}, that is,
$$
t(S) = \sum_{i\in S}\ y(\pi ,i)\qquad
\mbox{for every $S\in {\cal C}_{\pi}$}.
$$
Therefore, the definition of a desired game $t$ with $y=x^{t}$ is correct if and only if the
following {\em consistency condition\/} is satisfied:
\begin{equation}
\forall\, \tau ,\pi\in\Upsilon \quad \forall\, S\in {\cal C}_{\tau}\cap {\cal C}_{\pi}\qquad
\sum_{i\in S}\ y(\tau ,i) = \sum_{i\in S}\ y(\pi ,i)\,.
\label{eq.consistency}
\end{equation}
To verify \eqref{eq.consistency} realize that, for  $\tau ,\pi\in\Upsilon$ and
$S\in {\cal C}_{\tau}\cap {\cal C}_{\pi}$, \eqref{eq.chain-2} gives
${\cal C}_{\tau}\subseteq {\cal S}^{m}_{\tau}= {\cal S}_{\tau}$ and
${\cal C}_{\pi}\subseteq {\cal S}^{m}_{\pi}= {\cal S}_{\pi}$ and then
the condition (b) for $y$ implies $\sum_{i\in S} y(\tau ,i) = \sum_{i\in S} y(\pi ,i)$,
which was desired.

The second step is to verify that $t$ is standardized. Since $m$ is standardized, for any $\pi\in\Upsilon$,
one has $x^{m}(\pi ,\pi (1))=m(\{\pi (1)\})=0$, implying $\pi (1)\in N_{\pi}$. Then the condition (a) for $y$
implies $y(\pi ,\pi (1))=0$, that is, $t(\{\pi (1)\})=0$. In particular,
$t(S)=0$ for any $S\subseteq N$ with $|S|\leq 1$ and we know $t\in\stagame$.

The third step is to consider the line $L$ in $\stagame$ passing through $t$ and $m$, namely the collection of vectors
$$
q_{\varepsilon} := (1-\varepsilon)\cdot m + \varepsilon\cdot t\qquad
\mbox{where $\varepsilon\in {\dv R}$},
$$
and show that, for sufficiently small $\varepsilon$, one has $q_{\varepsilon}\in\supmogame$.
Since the payoff-array transformation is linear, for any $\varepsilon\in {\dv R}$, it transforms $q_{\varepsilon}$ to
$$
z_{\varepsilon} := (1-\varepsilon)\cdot x^{m} + \varepsilon\cdot y\,,
$$
that is, $L$ is transformed to the line in ${\dv R}^{\Upsilon\times N}$ passing through $y$ and
$x^{m}$. The condition \eqref{eq.chain} applied to elements of $L$ gives
$$
\forall\,\varepsilon\in {\dv R}\quad
\forall\,\pi\in\Upsilon\qquad
S\in {\cal C}_{\pi} ~\Rightarrow~ \sum_{i\in S}\, z_{\varepsilon}(\pi,i) =q_{\varepsilon}(S)\,.
$$
Further considerations are done with a fixed set $S\subseteq N$. One can certainly find and fix
$\pi\in\Upsilon$ with $S\in {\cal C}_{\pi}$. By \eqref{eq.chain-2}, one also has
$S\in {\cal S}^{m}_{\pi}= {\cal S}_{\pi}$. Since the conditions (a)-(b) define a
linear space in ${\dv R}^{\Upsilon\times N}$ and both $x^{m}$ and $y$ satisfy them, for any $\varepsilon\in {\dv R}$,
the vector $z_{\varepsilon}$ must satisfy them as well. Thus, the condition
(b), applied to $z_{\varepsilon}$, allows one to derive
$$
\forall\,\varepsilon\in {\dv R}\quad
\forall\,\tau\in\Upsilon ~
\mbox{with $S\in {\cal S}_{\tau}$}\qquad  \mbox{one has}~~
\sum_{i\in S}\, z_{\varepsilon}(\tau,i) =\sum_{i\in S}\, z_{\varepsilon}(\pi,i) =q_{\varepsilon}(S)\,.
$$
Now, consider $\tau\in\Upsilon$ with $S\not\in {\cal S}_{\tau}= {\cal S}^{m}_{\tau}$ instead.
By \eqref{eq.non-neg} combined with the definition of ${\cal S}^{m}_{\tau}$
and then by \eqref{eq.chain} applied to $m$ we get
$$
0<\sum_{i\in S} x^{m}(\tau,i)-m(S)=\sum_{i\in S}\, x^{m}(\tau,i)-\sum_{i\in S}\, x^{m}(\pi,i)\,.
$$
This allows one to write, for every $\varepsilon\in {\dv R}$, by \eqref{eq.chain} applied to $q_{\varepsilon}$,
\begin{eqnarray*}
\lefteqn{\sum_{i\in S}\, z_{\varepsilon}(\tau,i)- q_{\varepsilon}(S) =
\sum_{i\in S}\, z_{\varepsilon}(\tau,i)-\sum_{i\in S}\, z_{\varepsilon}(\pi,i)}\\
&=&
(1-\varepsilon )\cdot \left(\sum_{i\in S}\, x^{m}(\tau,i)-\sum_{i\in S}\, x^{m}(\pi,i)\right)
+\varepsilon\cdot \left(\sum_{i\in S}\, y(\tau,i)-\sum_{i\in S}\, y(\pi,i)\right)\,,
\end{eqnarray*}
and observe that the limit of this expression with $\varepsilon$ tending to zero is
positive. Therefore, for sufficiently small $|\varepsilon |$, one has the following:
$$
\forall\,\tau\in\Upsilon ~
\mbox{with $S\not\in {\cal S}_{\tau}$}\qquad
\mbox{one has} ~~ \sum_{i\in S}\, z_{\varepsilon}(\tau,i) >q_{\varepsilon}(S)\,.
$$
In particular, for sufficiently small $|\varepsilon |$, one has
$$
q_{\varepsilon}(S) = \min_{\tau\in\Upsilon}\, \sum_{i\in S}\, z_{\varepsilon}(\tau,i)\,,
$$
and, since this consideration can be done for any $S\subseteq N$, one can observe that
the condition \eqref{eq.min-weber} holds for $q_{\varepsilon}$
for sufficiently small $|\varepsilon |$, that is, $q_{\varepsilon}\in\supmogame$
by Lemma \ref{lem.1}.

Thus, there exists $0<\varepsilon$ such that both $r:=(1-\varepsilon )\cdot m+\varepsilon\cdot t$ and
$s:=(1+\varepsilon)\cdot m-\varepsilon\cdot t$ belong to $\supmogame$. Clearly,
$m=\frac{1}{2}\cdot r+ \frac{1}{2}\cdot s$. The line $L$ does not contain the zero vector $0$, as otherwise,
by linearity of the payoff-array transformation, one derives a contradictory conclusion $y\in\lin (x^{m})$.
Hence, $r$ and $s$ are non-zero. The fact $0\not\in L$ also gives the observation $\lin (r)\neq\lin (s)$.
Altogether, the condition \eqref{eq.non-extreme} has been verified.
\end{proof}

We show now that the removal of repeated row-occurrences in the array from Theorem~\ref{thm.1}
has no influence. Observe that any array $x\in {\dv R}^{\Gamma\times N}$ of the form \eqref{eq.gamma-array}
satisfies
\begin{equation}
\forall\,\tau\in\Gamma \quad
\exists\,\sigma\in\Upsilon \qquad
{\cal C}_{\sigma}\subseteq {\cal S}_{\tau}\,,
\label{eq.chain-3}
\end{equation}
which follows from \eqref{eq.chain-2} using Corollary \ref{cor.vert-core-supermod}.

\begin{lem}\label{lem.aux2}\rm
Under the assumptions of Theorem \ref{thm.1}, 
consider $\Gamma^{\prime}\subseteq\Gamma$ such that
\begin{itemize}
\item[(i)] $\forall\,\pi ,\tau\in\Gamma^{\prime}\qquad x(\pi,\ast)\neq x(\tau ,\ast)$,
\item[(ii)] $\forall\,\tau\in\Gamma\setminus\Gamma^{\prime} \quad \exists\,\pi\in\Gamma^{\prime}\qquad x(\pi,\ast)=x(\tau ,\ast)$.
\end{itemize}
Then, one can replace $\Gamma$ by $\Gamma^{\prime}$ in the condition from Theorem \ref{thm.1}.
\end{lem}

\begin{proof}
Assuming the condition for $\Gamma\times N$-array holds we observe easily that its restriction
to $\Gamma^{\prime}\times N$ satisfies it relative to $\Gamma^{\prime}\times N$. Indeed, if $y\in {\dv R}^{\Gamma^{\prime}\times N}$ satisfies (a)-(b), its extension based on (ii) satisfies (a)-(b) with respect to $\Gamma\times N$:
for $\tau\in\Gamma\setminus\Gamma^{\prime}$, we choose (and fix) $\pi\in\Gamma^{\prime}$ with $x(\pi,\ast)=x(\tau ,\ast)$ and put $y(\tau ,\ast):=y(\pi,\ast)$. The extension must be a multiple of (extended) $x$
and the same holds for their restrictions.

Conversely, assuming the condition from Theorem \ref{thm.1} holds for
$\Gamma^{\prime}\times N$-array, we verify it for $\Gamma\times N$.
If $y\in {\dv R}^{\Gamma\times N}$ satisfies (a)-(b), its restriction to
$\Gamma^{\prime}\times N$ satisfies them with respect to $\Gamma^{\prime}\times N$ and must be a
$\alpha$-multiple of the respective restriction of $x$, for some $\alpha\in {\dv R}$.
By (ii), for any $\tau\in \Gamma\setminus\Gamma^{\prime}$, we find (and fix) $\pi\in\Gamma^{\prime}$
such that $x(\pi,\ast)=x(\tau ,\ast)$, which implies
${\cal S}_{\tau}={\cal S}_{\pi}$. By \eqref{eq.chain-3}, choose
$\sigma\in\Upsilon$ with ${\cal C}_{\sigma}\subseteq {\cal S}_{\tau}$ and,
by the condition (b) for $y$, get
$$
\forall\, S\in {\cal C}_{\sigma}\subseteq {\cal S}_{\tau}={\cal S}_{\pi}\qquad
\sum_{i\in S}\, y(\tau ,i)=\sum_{i\in S}\, y(\pi ,i)\,,
$$
which allow one to derive, by inductive consideration, that $y(\tau ,\ast)=y(\pi ,\ast)$.
Therefore, $y$~must coincide with $\alpha\cdot x$.
\end{proof}

In conclusion, Theorem \ref{thm.1} now follows from Lemma \ref{lem.2}, Corollary \ref{cor.vert-core-supermod}
and Lemma \ref{lem.aux2}.

\section{Relation to generalized permutohedra}\label{sec.permutohedra}

Relatively recently, a highly relevant concept of a generalized permutohedron
has been introduced and studied by Postnikov and his co-authors \cite{post05,post08}.
The following is a minor modification of \cite[Definition~3.1]{post08}.

\begin{defin}[generalized permutohedron]\label{def.Gperm}\rm ~\\
Let $\{v_{\pi}\}_{\pi\in\Upsilon}$ be a collection of vectors in ${\dv R}^{N}$ parameterized
by enumerations (of $N$) such that for every $\pi\in\Upsilon$ and for every adjacent transposition
$\sigma :\ell\leftrightarrow \ell +1$, where $1\leq\ell <n$, a non-negative constant $k_{\pi,\ell}\geq 0$ exists such that
\begin{equation}
v_{\pi}-v_{\pi\compo\sigma} = k_{\pi,\ell}\cdot (\inc_{\pi (\ell )}- \inc_{\pi (\ell +1)})\,,
\label{eq.gener-permut}
\end{equation}
where $\pi\compo\sigma$ denotes the composition of $\pi$ with $\sigma$ and $\inc_{i}\in {\dv R}^{N}$ is the zero-one
identifier of a~variable $i\in N$ (see \S\,\ref{ssec.notation}).  The respective {\em generalized permutohedron\/} is then the convex hull of that collection of vectors:
$$
G(\{v_{\pi}\}_{\pi\in\Upsilon}) := \conv (\{ v_{\pi}\in {\dv R}^{N} \, :\ \pi\in\Upsilon\})\,.
$$
\end{defin}

\begin{example}\label{exa.permutohedron}
An example of a generalized permutohedron is the ``classic" permutohedron determined,
for example, by a strictly decreasing sequence of reals $r_{1}>r_{2}>\ldots >r_{n}$
as the convex hull $Q_{0}:=P(r_{1},r_{2},\ldots ,r_{n})$ of the collection of vectors
$$
v_{\pi} := [r_{\pi_{-1}(i)}]_{i\in N} \qquad \mbox{for $\pi\in\Upsilon$}\,.
$$
Indeed, if $\sigma :\ell\leftrightarrow \ell +1$, $1\leq\ell <n$ is an
adjacent transposition and $\pi \colon \{ 1,\ldots ,n\} \to N$ an enumeration
with $\pi (\ell )=a$, $\pi (\ell +1)=b$ then
$\pi_{-1} (i)=(\pi\compo\sigma )_{-1}(i)$ for $i\in N\setminus \{a,b\}$ and
\begin{eqnarray*}
v_{\pi} - v_{\pi\compo\sigma} &=&
r_{\pi_{-1}(a)}\cdot\inc_{a} + r_{\pi_{-1}(b)}\cdot\inc_{b} -
r_{(\pi\compo\sigma)_{-1}(a)}\cdot\inc_{a} - r_{(\pi\compo\sigma)_{-1}(b)}\cdot\inc_{b}\\
&=& r_{\ell}\cdot\inc_{a} + r_{\ell +1}\cdot\inc_{b} -
r_{\ell +1}\cdot\inc_{a} - r_{\ell}\cdot\inc_{b} = (r_{\ell}-r_{\ell +1})\cdot (\inc_{a}-\inc_{b})\\
&=& \underbrace{(r_{\ell}-r_{\ell+1})}_{>0}\cdot (\inc_{\pi (\ell )}-\inc_{\pi (\ell +1)})\,,
\end{eqnarray*}
which means the constant $k_{\pi,\ell}\equiv r_{\ell}-r_{\ell +1}$ in \eqref{eq.gener-permut} is strictly positive in this case.\\
{\em Side-note:} We believe
there is a misprint in \cite[\S\,3.1]{post08} in the motivational text preceding their
Definition 3.1. Specifically, we think the authors intended $a_{1}>a_{2}>\ldots >a_{n}$ instead of $a_{1}<a_{2}<\ldots<a_{n}$
in \cite[p.\,215 below]{post08}. Indeed, that ``decreasing"
convention, which was implicitly used in the original manuscript \cite{post05}, leads to \eqref{eq.gener-permut} while the opposite ``increasing"
convention leads to $k_{\pi,\ell}\leq 0$ in \eqref{eq.gener-permut}, respectively to \eqref{eq.GP-reform} below.
\end{example}

The concepts of a generalized permutohedron and that of a core of a supermodular game
basically coincide. The relation is evident through the concept of the Weber set (see Definition \ref{def.weber-set}).

\begin{lem}\label{lem.gener-permut}\rm
A polytope $P\subseteq {\dv R}^{N}$ is a generalized permutohedron iff there exists
a supermodular game $m$ over $N$ such that $P=\web (m)$.
\end{lem}

\begin{proof}
The first observation is that a generalized permutohedron can be equivalently introduced
as the convex hull \mbox{$\conv (\{ x_{\tau}\in {\dv R}^{N} \, :\ \tau\in\Upsilon\})$} of a
set of vectors $\{ x_{\tau}\}_{\tau\in\Upsilon}$ in ${\dv R}^{N}$ such that
\begin{equation}
\forall\,\tau\in\Upsilon \quad
\forall\,\varsigma : l\leftrightarrow l +1,\, 1\leq l <n\quad
\exists\, K_{\tau,l}\geq 0
 \qquad x_{\tau}-x_{\tau\compo\varsigma} = K_{\tau,l}\cdot (\inc_{\tau (l +1)}- \inc_{\tau (l)})\,,
\label{eq.GP-reform}
\end{equation}
which is, technically, the formula \eqref{eq.gener-permut} in which the right-hand side is multiplied by $(-1)$.
This paradox has an easy explanation: the vectors $v_{\pi}$, $\pi\in\Upsilon$ can be re-indexed by inverse enumerations instead, one can put, for any $\tau\in\Upsilon$,
\begin{eqnarray*}
x_{\tau} := v_{\pi} &\quad&  \mbox{where $\pi= \tau\compo\rho$ and $\rho$ is the ``inverting" permutation on $\{1,2,\ldots ,n\}$}\\
&&\mbox{ given by~~ $\rho(k):= n+1-k$ ~~for any $k\in\{1,2,\ldots ,n\}$.}
\end{eqnarray*}
Of course, the convex hull is the same after the re-indexing, but \eqref{eq.gener-permut} turns into \eqref{eq.GP-reform}.

The second observation is that the condition \eqref{eq.GP-reform} implies the following
{\em consistency condition\/}, analogous to the condition \eqref{eq.consistency}, namely
\begin{equation}
\forall\, \tau ,\pi\in\Upsilon \quad \forall\, S\in {\cal C}_{\tau}\cap {\cal C}_{\pi}\qquad
\sum_{i\in S}\ x_{\tau}(i) = \sum_{i\in S}\ x_{\pi}(i)\,.
\label{eq.cons.Gperm}
\end{equation}
Indeed, whenever $S\subseteq N$ is fixed and $\tau,\pi\in\Upsilon$ are such that
$S\in {\cal C}_{\tau}\cap {\cal C}_{\pi}$ one has
$\bigcup_{k\leq s} \{\tau (k)\}=S=\bigcup_{k\leq s} \{\pi (k)\}$ where $s=|S|$.
Hence, there exists a sequence of adjacent transpositions $\varsigma : l\leftrightarrow l +1,\, 1\leq l <n$ satisfying either $l+1\leq s$ or $s<l$ and transforming $\tau$ successively into $\pi$. For every such transposition $\varsigma$ and any $\upsilon\in\Upsilon$
in the transformation sequence from $\tau$ to $\pi$ one has
$\sum_{i\in S}\ x_{\upsilon}(i) = \sum_{i\in S}\ x_{\upsilon\compo\varsigma}(i)$
by \eqref{eq.GP-reform}, which allows one to derive \eqref{eq.cons.Gperm}.

Thus, the relation \eqref{eq.cons.Gperm} makes it possible to define correctly
a game $m$ over $N$ by
\begin{equation}
m(S) := \sum_{i\in S}\, x_{\tau}(i)\qquad
\mbox{whenever $S\in {\cal C}_{\tau}$~ for some $\tau\in\Upsilon$}.
\label{eq.inver.Gperm}
\end{equation}
The third observation is that the condition \eqref{eq.GP-reform} even implies that $m$
given by \eqref{eq.inver.Gperm} is supermodular. It is enough to show (see \ref{sec.apex-CI}),
for any $Z\subseteq N$ and distinct $a,b\in N\setminus Z$ that
$$
\Delta m(a,b|Z) ~:=~ m(\{a,b\}\cup Z)+ m(Z) -m(\{a\}\cup Z)-m(\{b\}\cup Z)\geq 0\,.
$$
Indeed, given such a set $Z$ with $s=|Z|$ find $\tau\in\Upsilon$ such that $\bigcup_{k\leq s} \{\tau(k)\}=Z$,
$\tau (s+1)=a$ and $\tau (s+2)=b$. Then consider an adjacent transposition $\varsigma :  s+1\leftrightarrow s +2$
and observe that $Z,\{a\}\cup Z,\{a,b\}\cup Z\in {\cal C}_{\tau}$ while $Z,\{b\}\cup Z,\{a,b\}\cup Z\in {\cal C}_{\tau\compo\varsigma}$.
The condition \eqref{eq.GP-reform} gives
$x_{\tau}-x_{\tau\compo\varsigma} = K_{\tau,s+1}\cdot (\inc_{\tau (s+2)}- \inc_{\tau (s+1)})= K_{\tau,s+1}\cdot (\inc_{b}- \inc_{a})$.
In other words,
$$
x_{\tau\compo\varsigma}(a)=x_{\tau}(a)+K_{\tau,s+1},\enskip x_{\tau\compo\varsigma}(b)=x_{\tau}(b)-K_{\tau,s+1}\enskip
\mbox{and $x_{\tau\compo\varsigma}(i)=x_{\tau}(i)$~ for $i\in N\setminus \{a,b\}$}.
$$
Hence, by \eqref{eq.inver.Gperm}, $m(Z)=\sum_{i\in Z} x_{\tau}(i)$,
$m(\{a\}\cup Z)=\sum_{i\in Z} x_{\tau}(i)+x_{\tau}(a)=m(Z)+x_{\tau}(a)$,
$m(\{b\}\cup Z)=\sum_{i\in Z} x_{\tau\compo\varsigma}(i)+x_{\tau\compo\varsigma}(b)=
m(Z)+x_{\tau\compo\varsigma}(b)=m(Z)+x_{\tau}(b)-K_{\tau,s+1}$
and, finally,
$m(\{a,b\}\cup Z)=\sum_{i\in Z} x_{\tau}(i)+x_{\tau}(a)+x_{\tau}(b)=m(Z)+x_{\tau}(a)+x_{\tau}(b)$. This gives
$$
m(\{a,b\}\cup Z)+ m(Z) -m(\{a\}\cup Z)-m(\{b\}\cup Z)= K_{\tau,s+1}\geq 0\,,
$$
which was desired. The comparison of \eqref{eq.inver.Gperm} and \eqref{eq.chain} gives
$x^{m}(\tau ,i)=x_{\tau}(i)$ for $i\in N$, $\tau\in\Upsilon$, which concludes the proof that
every generalized permutohedron is the Weber set for some supermodular game.

The converse implication saying that the Weber set $\web (m)$ for a supermodular game
$m$ over $N$ is a~generalized permutohedron is easier. It is enough to show that the
row-vectors of the $\Upsilon\times N$-array $x:= x^{m}$ given by \eqref{eq.payoff-array}
satisfy \eqref{eq.GP-reform}.
This can be verified by an inverse consideration: one can show that, for given $\tau\in\Upsilon$ and $\varsigma : l\leftrightarrow l +1,\, 1\leq l <n$
the condition \eqref{eq.GP-reform} holds with $K_{\tau,l}=m(\{a,b\}\cup Z)+ m(Z) -m(\{a\}\cup Z)-m(\{b\}\cup Z)$
where $Z=\bigcup_{k<l} \{\tau (k)\}$, $a=\tau (l)$ and $b=\tau (l+1)$. This concludes the proof.
\end{proof}

\begin{remark}\label{rem.Postnikov}\rm
Note that it follows from the above arguments that a polytope $P\subseteq {\dv R}^{N}$ is a
generalized permutohedron iff its $(-1)$-multiple is a generalized permutohedron. Indeed, if
$P=G(\{ v_{\pi}\}_{\pi\in\Upsilon})$ then $-P=G(\{ -v_{\pi}\}_{\pi\in\Upsilon})$ and
the vectors $v_{\pi}$, $\pi\in\Upsilon$ satisfy \eqref{eq.gener-permut} iff the vectors
$x_{\pi} := -v_{\pi}$, $\pi\in\Upsilon$ satisfy \eqref{eq.GP-reform}, with the same constant.
The first observation in the proof of Lemma \ref{lem.gener-permut} then implies what is claimed.

In particular, another equivalent formulation of Lemma \ref{lem.gener-permut} is that
a polytope $P\subseteq {\dv R}^{N}$ is a generalized permutohedron iff there exists
a {\em submodular\/} game $m$ over $N$ such that $P=\web (m)$. This is because our payoff-array transformation is
linear: therefore, $\web (-m)=(-1)\cdot \web (m)$ and we know $m$ is supermodular iff $-m$ is submodular.
\end{remark}

The consequence of Lemma \ref{lem.gener-permut} and Lemma \ref{lem.1} (the second statement) is as follows.

\begin{coro}\label{cor.GP-supermod-core}\rm
A polytope $P\subseteq {\dv R}^{N}$ is a generalized permutohedron iff
it is the core of a~{\em supermodular\/} game $m$ over $N$,
that is, iff $\exists\, m\in\supermo$ such that
$$
P=\cor (m)\equiv \left\{\, [v_{i}]_{i\in N}\in {\dv R}^{N}\, : ~ \sum_{i\in N} v_{i}= m(N) \,~\&~\,
\forall\, S\subseteq N ~\sum_{i\in S} v_{i}\geq m(S)\,\right\}\,.
$$
\end{coro}

Thus, the class of generalized permutohedra coincides with the class of cores of ``convex" (= supermodular) games.
The dual formulation of Corollary \ref{cor.GP-supermod-core} is that $P$ is a~generalized permutohedron iff there exists a submodular game
$r\in {\dv R}^{\caP}$, $r(\emptyset )=0$ with
$$
P=\left\{\, [v_{i}]_{i\in N}\in {\dv R}^{N}\, : ~ \sum_{i\in N} v_{i}= r(N) \,~\&~\,
\forall\, S\subseteq N ~\sum_{i\in S} v_{i}\leq r(S)\,\right\}\,.
$$
Indeed, the relation between the lower bounds from Corollary \ref{cor.GP-supermod-core} and
the upper bounds in the above formula is as follows: $r(S)+m(N\setminus S)=m(N)= r(N)$ for any $S\subseteq N$. Note that this is one of possible correspondences between supermodular and submodular games, see the discussion in \S\,\ref{ssec.matroid}, the relations \eqref{eq.r-to-m} and \eqref{eq.m-to-dual-r} on page~\pageref{eq.r-to-m}.

\begin{remark}\label{rem.Doker}\rm
Note that the fact that every generalized permutohedron has the form \eqref{eq.def.core}
has also been mentioned in recent literature on generalized permutohedra.
Nevertheless, we feel that the formulation of this fact in \cite[\S\,2]{ABD11} and
\cite[\S\,2.2]{dok11} is somehow ambiguous and needs clarification or a
warning of possible misinterpretation. Since the formal definition of the concept
of a generalized permutohedron is omitted therein,\footnote{An informal vague sentence with a reference
to \cite{post08} is only written there instead.}
the reader of \cite{ABD11,dok11}, not being aware of the precise definition, easily gets the impression that
generalized permutohedra are ``just" defined by \eqref{eq.def.core}, where the lower bounds $m(S)$, $S\subseteq N$
are required to be tight. Then subsequent Theorem 2.1 in \cite{ABD11}, respectively Theorem 2.2.1 in
\cite{dok11}, may be misinterpreted as the claim that the lower bounds in \eqref{eq.def.core}
are tight iff they define a supermodular game. This is not true as Example \ref{exa.Doker} showed;
the polytope there {\em is not a generalized permutohedron}, despite it is the core
of an exact game.
\end{remark}

To explain the relation of our result from \S\,\ref{ssec.main-formul} to the theory of
generalized permutohedra let us mention their equivalent characterization in terms of Minkowski sum.
It is based on the following concept recalled in \cite[\S\,1.1]{mort07}.

\begin{definition}\label{def.Mink-summand}\rm
We say that a polytope $P\subseteq {\dv R}^{N}$ is a {\em Minkowski summand\/}
of a polytope $Q\subseteq {\dv R}^{N}$ if there exists $\lambda >0$ and a polytope
$R\subseteq {\dv R}^{N}$ such that $\lambda\cdot Q=P\oplus R$.
\end{definition}

The following characterization of generalized permutohedra has been given in Proposition 3.2 of \cite{post08};
see also related Theorem 2.4.3 in \cite{mort07}.

\begin{lem}\label{lem.GP-Mink-sum}\rm
A polytope $P\subseteq {\dv R}^{N}$ is a generalized permutohedron iff it is
a Minkowski summand of the permutohedron.
\end{lem}

\begin{proof}
We can easily show that every generalized permutohedron is a summand of the permutohedron.
To this end we define a special standardized game $\overline{m}$ by
$$
\overline{m}(S)= \frac{1}{2}\cdot |S|\cdot (|S|-1)\qquad
\text{for any $S\subseteq N$.}
$$
Observe that its Weber set $Q_{0}:= \web (\overline{m})$ coincides with the permutohedron
$P(r_{1},\ldots ,,r_{n})$, where $r_{k}:= n-k$ for $k= 1,\ldots ,n$ (see Example \ref{exa.permutohedron}).
Note that, for any $Z\subseteq N$ and distinct $a,b\in N\setminus Z$, one has
$$
\Delta \overline{m}(a,b|Z) ~=~\overline{m}(\{ a,b\}\cup Z)+\overline{m}(Z)-\overline{m}(\{ a\}\cup Z)-\overline{m}(\{ b\}\cup Z)=1\,,
$$
which allows one to observe (see \ref{sec.apex-CI}) that $\overline{m}\in\supmogame$ and,
for any $\widetilde{m}\in\supmogame$, there exists $\lambda>0$ such that
$\lambda\cdot\overline{m}-\widetilde{m}\in\supmogame$.

Thus, given a generalized permutohedron $P\subseteq {\dv R}^{N}$, by Lemma \ref{lem.gener-permut}
we find a supermodular game $m$ such that $P=\web (m)$ and define a standardized version
$m^{\star}$ of $m$ by \eqref{eq.standard}.
Find $\lambda>0$ with $\lambda\cdot\overline{m}-m^{\star}\in\supmogame$ and observe
$$
\lambda\cdot\overline{m} = m+r\quad
\mbox{where $r:=(\lambda\cdot\overline{m}-m^{\star})-\sum_{i\in N} m(\{ i\})\cdot m^{\uparrow i}$~ is a supermodular game.}
$$
Hence, $\lambda\cdot Q_{0}=\web(\lambda\cdot\overline{m}) =\web (m)\oplus \web (r)=P\oplus \web (r)$.

For the inverse implication we refer to Proposition 3.2 and Theorem 15.3 in \cite{post08}. The arguments there
go through another equivalent definition of a generalized permutohedron saying that its normal fan
coarsens the normal fan of the permutohedron; for these concepts see \S\,1.1 of \cite{mort07}.
\end{proof}

Thus, one can introduce a natural pre-order on the class of generalized permutohedra, namely $P\preceq Q$ iff
$P$ is a Minkowski summand of (a generalized permutohedron) $Q$ and the respective equivalence relation
$P\simeq Q$ defined by $P\preceq Q\preceq P$. This leads to the following concept motivated by the general notion
of join-irreducibility from lattice theory; see \cite[\S\,III.3]{bir91}. Moreover, in our context,
this concept also appears to correspond to the notion of an indecomposable polytope; see Remark \ref{rem.indec-polytope}.
This fact motivated our terminology.

\begin{definition}\label{def.ir-Mink-summand}\rm
A generalized permutohedron $P\subseteq {\dv R}^{N}$ has a non-trivial decomposition if
$$
\lambda\cdot P = P_{1}\oplus \ldots \oplus P_{k}\qquad
\mbox{for some $\lambda>0$ and generalized permutohedra $P_{1},\ldots ,P_{k}$,}
$$
none of which is equivalent to $P$. If this is not the case we say that $P$ is {\em indecomposable}.
\end{definition}

One anticipates that any generalized permutohedron can be decomposed into indecomposable ones. Therefore,
a natural question is whether one can geometrically characterize
the indecomposable generalized permutohedra. A trivial observation is that $P\subseteq {\dv R}^{N}$
is indecomposable iff any translation $P\oplus \{ v\}$ for $v\in {\dv R}^{N}$
is indecomposable. Therefore, we are only interested in
{\em standardized\/} polytopes, that is, polytopes  $P\subseteq [0,\infty )^{N}$
such that, for any $i\in N$, an element $v\in P$ exists with $v_{i}=0$. Our result
from \S\,\ref{ssec.main-formul} can be interpreted
as the solution to the problem of characterization of indecomposable generalized permutohedra.

\begin{thm}\label{thm.ired}
A standardized generalized permutohedron $P\subseteq {\dv R}^{N}$ is
indecomposable (in sense of Definition~\ref{def.ir-Mink-summand}) iff the set ${\cal X}$ of its vertices satisfies the condition of Theorem \ref{thm.1},
that is, given $x\in {\dv R}^{\Gamma\times N}$ satisfying \eqref{eq.gamma-array}, every solution
$y\in {\dv R}^{\Gamma\times N}$ to (a)-(b) is a multiple of $x$.
If this is the case, then the only standardized generalized permutohedra equivalent to $P$ are its multiples $\lambda\cdot P$ where $\lambda>0$.

For any non-empty finite set of variables $N$, there exists a finite number of indecomposable types of generalized
permutohedra. Every generalized permutohedron can be written as the Minkowski sum of indecomposable ones.
\end{thm}

\begin{proof}
First, observe that a generalized permutohedron $P\subseteq {\dv R}^{N}$ is a summand of a~generalized permutohedron $Q\subseteq {\dv R}^{N}$ iff
there exists another generalized permutohedron $R\subseteq {\dv R}^{N}$ such that $\lambda\cdot Q=P\oplus R$ for some $\lambda>0$.
The sufficiency of this condition is evident. For necessity write $\lambda\cdot Q=P\oplus R$, where $\lambda >0$ and
$R\subseteq {\dv R}^{N}$ is a polytope. To show that $R$ is a generalized permutohedron Lemma \ref{lem.GP-Mink-sum}
can be used. This lemma, applied to $Q$, says there exists $\gamma >0$ and a polytope $R^{\prime}\subseteq {\dv R}^{N}$
such that $\gamma\cdot Q_{0}= Q\oplus R^{\prime}$, where $Q_{0}$ is the usual permutohedron (see Example \ref{exa.permutohedron}). Hence,
$$
(\lambda\cdot\gamma)\cdot Q_{0} = (\lambda\cdot Q)\oplus (\lambda\cdot R^{\prime})=
P\oplus R\oplus (\lambda\cdot R^{\prime})= R\oplus (P\oplus (\lambda\cdot R^{\prime}))\,,
$$
which, again by Lemma \ref{lem.GP-Mink-sum}, this time applied to $R$, says $R$ is a generalized permutohedron.

By Corollary \ref{cor.GP-supermod-core}, the mapping $m\mapsto \cor (m)=\web (m)=P$
considered on the cone of supermodular games is a mapping from $\supermo$ onto the set
of generalized permutohedra. By Lemma \ref{lem.1}, the mapping is one-to-one:
the inverse mapping is given by $m^{P}(S)=\min_{v\in P}\,\sum_{i\in S} v_{i}$ for $S\subseteq N$;
see also Theorem \ref{thm:super}\eqref{i:web-min}.

Since the mapping transforms sums to Minkowski sums and non-negative multiples to
non-negative multiples, any decomposition $\lambda\cdot Q=P\oplus R$, $\lambda >0$, corresponds
to a~decomposition of the respective game $\lambda\cdot m^{Q}=m^{P}+m^{R}$.
Clearly, standardized games corresponds to standardized polytopes.
If a standardized generalized permutohedron has a non-trivial decomposition then it has
a non-trivial decomposition into standardized polytopes.

A decomposition $\lambda\cdot m^{Q}=m^{P}+m^{R}$ into standardized games also implies that
$m^{P},m^{R}$ belong to the smallest face $F(m^{Q})$ of $\supmogame$ containing $m^{Q}$.
Conversely, if $m^{P}\in F(m^{Q})$ then $m^{\prime}\in F(m^{Q})$ exists with
$m^{Q}=\alpha\cdot m^{P}+(1-\alpha )\cdot m^{\prime}$ for some $\alpha\in (0,1)$.
Since there exists a standardized generalized permutohedron $R$ with
$\alpha^{-1}\cdot (1-\alpha )\cdot m^{\prime}=m^{R}$, one has
$\alpha^{-1}\cdot m^{Q}=m^{P}+m^{R}$ and concludes that $P\preceq Q$
iff $m^{P}\in F(m^{Q})$. Therefore, $P\simeq Q$ iff $F(m^{P})=F(m^{Q})$.

In particular, a non-trivial decomposition of $P$ in sense of Definition \ref{def.ir-Mink-summand}
correspond to a non-trivial decomposition of $\lambda\cdot m^{P}$ in the sense that
none of its summands in $\supmogame$ belongs to the relative interior of $F(m^{P})$.
Thus, a standardized generalized permutohedron $P\subseteq {\dv R}^{N}$ is indecomposable iff
$m^{P}$ belongs to an extreme ray of $\supmogame$. The generators of these rays are characterized
by Theorem \ref{thm.1}; the zero function also trivially satisfies the condition from \S\,\ref{ssec.main-formul}
because ${\cal X}$ has just the zero vector in that case. The remaining statements of Theorem \ref{thm.ired} then
follow from the fact that $\supmogame$ is a pointed rational polyhedral cone.
\end{proof}

\begin{remark}\label{rem.indec-polytope}\rm
A non-empty polytope $P\subseteq {\dv R}^{N}$ is called {\em indecomposable\/} if every
Minkowski summand of $P$ has the form $\alpha\cdot P\oplus \{ v\}$, where $\alpha\geq 0$ and
$v\in {\dv R}^{N}$. This notion, treated in the theory of convex polytopes \cite{mey74,GKKZ03},
aims at capturing the concept of extremality in an abstract way.
Specifically, the set of non-empty polytopes in ${\dv R}^{N}$, being equipped with the
Minkowski addition $\oplus$ and the scalar multiplication by non-negative reals,
can be viewed as an abstract convex cone. The polytope is then indecomposable when
it is an atom in the ``face lattice" of this abstract convex cone.

Given a generalized permutohedron $P\subseteq {\dv R}^{N}$, if there is a non-trivial
decomposition of $P$ in the sense of Definition~\ref{def.ir-Mink-summand}, then $P$ has a
{\em polytopal\/} decomposition in the sense $P=P_{1}\oplus \ldots \oplus P_{k}$, $k\geq 2$,
where at least one of $P_{i}$, $1\leq i\leq k$ is {\em not} in the form
$\alpha\cdot P\oplus \{ v\}$ with $\alpha\geq 0$ and $v\in {\dv R}^{N}$; note this is
a corrected version of the definition of a decomposable polytope from \cite[p.\,318]{GKKZ03}.
Indeed, such a polytope $P_{i}=\alpha\cdot P\oplus \{ v\}$ is either equivalent to $P$
(if $\alpha >0$) or a singleton. But the Minkowski sum of singletons is a singleton, which has no non-trivial decomposition.
In particular, every generalized permutohedron that is an indecomposable polytope is
indecomposable in the sense of Definition \ref{def.ir-Mink-summand}.

Nevertheless, the converse is true as well. To observe that assume for a contradiction that $P$
is a~generalized permutohedron indecomposable in the sense of Definition~\ref{def.ir-Mink-summand} but
not an indecomposable polytope. Then, by \cite[Theorem~4]{mey74}, a finite collection
$Q_{1},\ldots , Q_{k}$, $k\geq 2$ of indecomposable polytopes exists such that
$P=Q_{1}\oplus\ldots\oplus Q_{k}$; without loss of generality assume that both $P$ and
$Q_{1},\ldots , Q_{k}$ are standardized. Each $Q_{i}$, $1\leq i\leq k$ is a~Minkowski summand of $P$, and, therefore,
using Lemma~\ref{lem.GP-Mink-sum} applied to $P$, a summand of the classic permutohedron.
Thus, every $Q_{i}$ is a generalized permutohedron, again by Lemma~\ref{lem.GP-Mink-sum}.
Since $P$ itself is assumed to be indecomposable in the sense of Definition~\ref{def.ir-Mink-summand} the fact
$P=Q_{1}\oplus\ldots\oplus Q_{k}$ implies there exists $Q_{i}$, $1\leq i\leq k$ equivalent to $P$.
By the second claim of Theorem~\ref{thm.ired} applied to $Q_{i}$, $P$~must be a positive multiple of $Q_{i}$.
As $Q_{i}$ is an indecomposable polytope, the same holds for $P$, which is the contradiction.
\end{remark}

Remark \ref{rem.indec-polytope} means our Theorem \ref{thm.ired} can be interpreted as follows: given a~generalized
permutohedron $P\subseteq {\dv R}^{N}$, we provide a necessary and sufficient condition
on $P$ being an indecomposable polytope. The criterion we give to decide that question
is based on solving a particular system of homogenous linear equations.

\begin{remark}\label{rem.Meyer}\rm
Meyer in his 1974 paper \cite{mey74} gave a criterion to recognize whether a given
polytope $P$ is indecomposable, which is also based on solving a system of homogenous linear equations.
The reader can ask whether our condition from \S\,\ref{ssec.main-formul} is the special case of Meyer's criterion.
The answer is that the two criteria differ significantly in terms of methodology and motivation.

Specifically, the condition (2) of Theorem 3 in \cite{mey74} characterizes indecomposability
of a polytope $P\subseteq {\dv R}^{N}$ in terms of a linear equation system, denoted $e[P]$ there.
To apply that result one needs to have a complete list of facets of $P$ at disposal; note that, in our context
of generalized permutohedra, the facets correspond to (some of the) subsets of $N$. If $P$ is full-dimensional,
that is, if $\dim (P)=|N|=n$, then each equation in $e[P]$ corresponds to a certain set of $n+1$ facets of $P$
which intersect in a vertex of $P$; see \cite[p.~79--80]{mey74} where
the equation system is specified. The condition (2) in \cite[Theorem 3]{mey74} can equivalently
be stated that the dimension of the space of solutions to $e[P]$ is $n+1$.
Meyer also considers an extended equation system. His idea is to add further
$n$ standardization equations corresponding to the translation of $P$ so that its Steiner point,
which is a kind of barycenter, is the zero vector.
Then his criterion turns into the condition that the dimension of the space of solutions to
the extended linear equation system is just 1; this is basically the condition (1) in \cite[Theorem 3]{mey74}.
In this aspect, our condition from \S\,\ref{ssec.main-formul} is analogous.

However, in Meyer's equation system, the indeterminates correspond to facets of $P$, that is,
in the context of generalized permutohedra, to subsets of $N$. Thus, it is clear from this observation
and the fact that the indeterminates of our system are the pairs (vertex,variable) (see \S\,\ref{ssec.interpret})
that our system and Meyer's system are methodologically different. Moreover, our criterion does
not require computing the facets of $P$, although this is not a big problem in the case of a generalized permutohedron.

Another note is that, in the case of generalized permutohedra, the linear equations in the system $e[P]$ used
by Meyer \cite{mey74} seem to have similar form as the linear (in)equalities provided
by Kuiper {\em et al.\/} \cite{KVV10} discussed in Remark \ref{rem.Kuipers}.
\end{remark}

\section{Relation to a former result by Rosenm\"{u}ller and Weidner}\label{sec.RW}

The paper by Rosenm\"{u}ller and Weidner \cite{RW74} offers another criterion to recognize
extreme supermodular functions. The reader may be interested in what features our new result
is different from their old one, if there is a substantial difference at all.

The answer is that our criterion characterizing extreme supermodular functions is indeed different
from their criterion, although analogous in certain aspects.
The main difference is that our characterization comes from a {\em min-representation\/} of
a supermodular function by means of {\em additive functions},
while the characterization by Rosenm\"{u}ller and Weidner is based on a {\em max-representation\/}
of a standardized supermodular game by means of {\em modular functions}. Below we give
some subtle arguments in favour of the opinion that the min-representation of a supermodular function
is more natural than its max-representation. Therefore, our characterization may appear to be more convenient.

\subsection{Recalling the criterion by Rosenm\"{u}ller and Weidner}\label{ssec.Rosen-recall}

The main obstacle to compare transparently both criteria was that the paper \cite{RW74}
had been written in technically awkward style. What follows is a kind of re-interpretation of their result.
We provide simpler presentation of their result (than the original one) and this allows us to explain clearly
in what aspects our result is different and in what aspects the results are analogous.
\smallskip

The paper \cite{RW74} deals with non-negative supermodular games.
A simple consideration, made in Remark \ref{rem.exteme}, allows one to observe their paper also
gives a criterion to recognize the extreme rays of the cone $\supmogame$. Any such a game can be
written as the maximum of finitely many modular functions $l$ on $\caP$ of a special form, namely
$$
l(S) =\left(\,\sum_{i\in S}\ z_{i}\right) - z_{\emptyset}\quad \mbox{for $S\subseteq N$},
\quad \mbox{where $z_{\emptyset}\geq 0$ and $z_{i}\geq 0$ for $i\in N$}
$$
are non-negative coefficients. The {\em max-representation\/}
of $m\in\supmogame$ has the form
\begin{equation}
m(S) =\max_{\tau\in\Omega}\, l^{\tau}(S) := \max_{\tau\in\Omega}~ \left(\,- z^{\tau}_{\emptyset}+\sum_{i\in S} z^{\tau}_{i}\right) \qquad \mbox{for $S\subseteq N$},
\label{eq.max-repr}
\end{equation}
where $\Omega$ is a finite index set identifying the modular functions. For each $\tau\in\Omega$, the modular function $l^{\tau}$ is specified by a vector in ${\dv R}^{n+1}$ of its non-negative coefficients $z^{\tau}_{\emptyset}$ and $z^{\tau}_{i}$, $i\in N$.
The indexing of modular functions in \eqref{eq.max-repr}
plays only an auxiliary role, since each modular function
can be identified with the vector of its coefficients, viewed
as a row-vector in ${\dv R}^{\{\emptyset\}\cup N}$. Thus, the max-representation
of $m\in\supmogame$ can alternatively be described by
a real $\Omega\times (\{\emptyset\}\cup N)$-array of the respective (non-negative) coefficients.
\smallskip

Rosenm\"{u}ller and Weidner \cite{RW74} give further technical conditions on the representation
\eqref{eq.max-repr}, which allows them to introduce a unique {\em canonical representation} for each $m\in\supmogame$, up to re-indexing. Specifically, each modular function $l^{\tau}$, $\tau\in\Omega$ is ascribed its ``carrier" $C^{\tau}$ and the class of sets ${\cal Q}^{\tau}$ at which the max-representation \eqref{eq.max-repr} is tight:
$$
C^{\tau} := \{\, i\in N\,:\ z^{\tau}_{i}>0\,\},
\quad
{\cal Q}^{\tau} :=  \{\, S\subseteq N\,:\ l^{\tau}(S)=m(S)\,\}\qquad
\mbox{for every $\tau\in\Omega$}.
$$
Every $S\subseteq N$ is assigned a ``tuft" of its subsets differing from it in at most one element:
$$
\mbox{tuft}\,(S) := \{\, T\subseteq S\,:\ |S\setminus T|\leq 1\,\} = \{ S\}\cup \bigcup_{i\in S}\: \{\, S\setminus\{ i\} \,\}\,.
$$
The technical conditions are as follows:
\begin{itemize}
\item[(i)] $\forall\, \tau\in\Omega \quad \mbox{tuft}\,(C^{\tau})\subseteq {\cal Q}^{\tau}$,
\item[(ii)] $\forall\, S\subseteq N ~~\exists\, \tau\in\Omega \quad \mbox{tuft}\,(S)\subseteq {\cal Q}^{\tau}$,
\item[(iii)] $\forall\, \tau,\pi\in\Omega \qquad {\cal Q}^{\pi}\subseteq {\cal Q}^{\tau}
~\Rightarrow~ \pi =\tau$.
\end{itemize}
Note that (i)-(iii) is our formally weakened re-formulation of the conditions
(1)-(3) from Theorem~2.5 in \cite{RW74} which, however, leads to the same concept of a canonical max-representation.
The existence of a max-representation of $m\in\supmogame$ satisfying (i)-(iii) can be
shown as follows. One puts $\Omega =\caP$ and ascribes a modular function $l^{T}$
to every $T\in\Omega$ by defining directly its coefficients:
\begin{eqnarray}
z^{T}_{i} &:=& m(T)-m(T\setminus\{ i\}) ~~ \mbox{for $i\in N$}, \label{eq.RW-formula}\\
z^{T}_{\emptyset} &:=& -m(T)+\sum_{i\in T} z^{T}_{i} = (|T|-1)\cdot m(T)-\sum_{i\in T} m(T\setminus\{ i\})\,. \nonumber
\end{eqnarray}
Then \eqref{eq.max-repr} holds and (i)-(ii) are fulfilled.
Finally, $\Omega$ is reduced so that repeated occurrences of functions are removed and, to ensure (iii), the elements $T\in\Omega$
with non-maximal tightness set classes ${\cal Q}^{T}$ (with respect to inclusion) are dropped. We refer to \cite{RW74}
for the arguments why every max-representation of $m$ satisfying (i)-(iii)
has the above form.
\smallskip

A necessary and sufficient condition on $m$ to be extreme, called {\em non-degeneracy\/}
by Rosenm\"{u}ller and Weidner \cite{RW74}, is that a certain system of
linear equations on the elements of the above mentioned
$\Omega\times (\{\emptyset\}\cup N)$-array (given by the canonical max-representation) has a unique solution up to a real multiple.
Specifically, one can put
$$
{\cal Q}^{\tau}_{0} := {\cal Q}^{\tau}\cap \{\, S\subseteq N\, :\ m(S)=0\,\}=
\{\, S\subseteq N\, :\ 0=l^{\tau}(S)=m(S)\,\}\quad \mbox{for $\tau\in\Omega$}
$$
and consider the following system of linear constraints on a real array
$y\in{\dv R}^{\Omega\times (\{\emptyset\}\cup N)}$:
\begin{itemize}
\item[(x)] $\forall\, \tau\in\Omega \qquad \mbox{if\, $i\in N\setminus C^{\tau}$ then}
~~ y(\tau ,i) =0$,
\item[(y)] $\forall\, S\subseteq N ~~~\forall\, \tau ,\pi\in\Omega ~~\mbox{such that $S\in {\cal Q}^{\tau}\cap {\cal Q}^{\pi}$}$\\[0.6ex]
\hspace*{2cm} $-y(\tau ,\emptyset )+\sum_{i\in S}\, y(\tau ,i) = -y(\pi ,\emptyset )+\sum_{i\in S}\, y(\pi ,i)$,
\item[(z)] $\forall\, \tau\in\Omega ~~~\forall\, S\in {\cal Q}^{\tau}_{0} \qquad
-y(\tau ,\emptyset)+\sum_{i\in S}\, y(\tau ,i) =0$.
\end{itemize}
Again, (x)-(z) is simplified, but equivalent, formulation of the conditions (3.1)-(3.3) from \cite{RW74}.
The starting array $y(\tau ,j)=z^{\tau}_{j}$ for $\tau\in\Omega$ and $j\in \{\emptyset\}\cup N$
given by \eqref{eq.RW-formula} is a solution to (x)-(z). The main result of \cite{RW74} is that
$m$ is extreme iff the system of linear constraints (x)-(z) has a unique solution up to a multiple constant.

Observe that the condition of non-degeneracy is analogous to our condition from
\S\,\ref{ssec.main-formul}: in both cases the requirement is that a solution to
a system of linear constraints on the element of the respective real array is unique
up to a multiple. Even the conditions are analogous: (a) corresponds to (x), (b) to
(y) and (z) is a further condition forced by the presence of an additional component
in the max-representation. On the other hand, the rows in the arrays do not correspond to each other:
in case of our condition from \S\,\ref{ssec.main-formul} they correspond to additive
upper bounds for $m$, while in case of the non-degeneracy condition they describe modular lower bounds
for $m$. The relation is illustrated by a~simple example.

\begin{example}\label{exa.RW-compar}
Assume $N=\{a,b,c\}$ and consider the game $m$ over $N$ given by
$$
m=2\cdot\delta_{N} +  \delta_{\{a,b\}} + \delta_{\{a,c\}} + \delta_{\{b,c\}}.
$$
In Example \ref{exa.positive}, we have verified that $m$ generates an extreme ray of $\supmogame$
using our criterion from \S\,\ref{ssec.main-formul}, based on the min-representation of $m$.
As concerns the max-representation, the above described procedure based on \eqref{eq.RW-formula}
results in five different modular functions, given by the following vectors
$[z_{\emptyset} \,|\, z_{a},z_{b},z_{c}]$ of coefficients:
$$
[0\,|\,0,0,0],\quad [1\,|\,1,1,0],\quad [1\,|\,1,0,1],\quad [1\,|\,0,1,1],\quad [1\,|\,1,1,1]\,.
$$
The middle three of them can be dropped because they do not give maximal tightness set classes.
The canonical max-representation can be arranged into an
$\Omega\times (\{\emptyset\}\cup N)$-array with $\Omega =\{\mu ,\nu\}$:
$$
\bbordermatrix{
   & \emptyset & \VR a & b & c \cr
 \mu & 0 & \VR 0 & 0 &ÃÂ 0 \cr
 \nu & 1 & \VR 1 & 1 & 1 \cr }\,.
$$
Table \ref{tab.0} indicates by bullets and by checkmarks what are the corresponding tightness sets
in ${\cal Q}^{\tau}\setminus {\cal Q}^{\tau}_{0}$ and in ${\cal Q}^{\tau}_{0}$, for $\tau\in\Omega$, respectively.
Now, solving (x)-(z) first gives $y(\mu ,a)=y(\mu ,b)=y(\mu ,c)=0$ by (x) and then
$y(\mu ,\emptyset)=0$ by (z). Then $\{a\}\in {\cal Q}^{\nu}\cap {\cal Q}^{\mu}$
gives by (y) $y(\nu, a)- y(\nu,\emptyset)=y(\mu, a)- y(\mu,\emptyset)=0$ and, analogously,
$y(\nu,\emptyset)=y(\nu ,a)=y(\nu ,b)=y(\nu ,c)$. Thus, $m$ is extreme by the non-degeneracy criterion
of Rosenm\"{u}ller and Weidner. Note that (x)-(z) has a unique solution up to a constant
even if we do not drop the modular functions with non-maximal tightness set classes from the
``starting" max-representation with five rows.
\end{example}

\begin{table}
$$
\begin{array}{lcccccccc}
 &\emptyset & \{a\} & \{b\} & \{c\} &\{a,b\} & \{a,c\} &\{b,c\} & N\\ \cline{2-9}
\mu &\checkmark & \checkmark & \checkmark & \checkmark &&&&\\
\nu && \checkmark & \checkmark & \checkmark &\bullet  &\bullet  &\bullet  &\bullet \\
\end{array}
$$
\caption{The tightness sets in the max-representation from Example \ref{exa.RW-compar}.}\label{tab.0}
\end{table}

\subsection{The min-representation versus the max-representation}\label{Rosen-compare}

The class of supermodular games is neither closed under minimization nor closed under maximization.
Indeed, in the case $N=\{ a,b,c\}$ one has
$$
\delta_{N}+\delta_{\{a,b\}}+\delta_{\{a,c\}}= \max\, \left\{
\delta_{N}+\delta_{\{a,b\}} \,,\,
\delta_{N}+\delta_{\{a,c\}}
\right\},
$$
while in the case  $N=\{ a,b,c,d\}$ one has
\begin{eqnarray*}
\lefteqn{3\cdot\delta_{N}+2\cdot\delta_{\{a,b,c\}}+2\cdot\delta_{\{a,b,d\}}}\\
&=& \min\, \left\{
4\cdot\delta_{N}+2\cdot\delta_{\{a,b,c\}}+2\cdot\delta_{\{a,b,d\}} \,,\,
3\cdot\delta_{N}+2\cdot\delta_{\{a,b,c\}}+2\cdot\delta_{\{a,b,d\}}+\delta_{\{a,b\}}
\right\}.
\end{eqnarray*}
Therefore, none of these two representation modes has plain advantage.
\smallskip

Nevertheless, there are fine arguments in favor of the min-representation.
They apply once one decides to interpret a set function as a function defined on
the vertices of the hypercube and to consider its extensions. Indeed, one can embed $\caP$ into $[0,1]^{N}$
using the mapping $S\mapsto \incS_{S}$ for $S\subseteq N$; see \eqref{eq.incidence}.
The core-based min-representation of $m\in\supmogame$ is related to the concept of the {\em Lov\'{a}sz extension} as defined in \cite{lov83}.
Indeed, \eqref{eq.min-weber} has the form
$$
m(S)=\min_{\tau\in\Upsilon}\, \sum_{j\in N}\, x^{\tau}_{j}\cdot \incS_{S}(j)
\quad \mbox{where $S\subseteq N$ and $x^{\tau} := [x^{m}(\tau ,i)]_{i\in N}$ for $\tau\in\Upsilon$.}
$$
This naturally leads to the extension $\lov$ of $m$ on the hypercube:
\begin{equation}
\lov (y) := \min_{\tau\in\Upsilon}\, \sum_{j\in N} x^{\tau}_{j}\cdot y_{j}
\qquad
\mbox{for $y\in [0,1]^{N}$.}
\label{eq.min-Lov-ext}
\end{equation}
Since pointwise minimum of affine functions is concave, $\lov$ is a {\em concave function\/} on $[0,1]^{N}$.
It follows from the basic facts about the Lov\'{a}sz extension $\widehat{m}$ (see \ref{sec.apex-supermod})
that verifying $\lov=\widehat{m}$ over $[0,1]^{N}$ boils down to show that, for any enumeration
$\pi\in\Upsilon$, $\lov$ is a linear function on the simplex\footnote{In this section,
by a {\em simplex} we understand the convex hull of an affinely independent set of vectors.}
$$
\nabla_{\pi} := \conv \{\, \incS_{S}\, :\ S\in {\cal C}_{\pi} \} \subseteq [0,1]^{N}\,.
$$
Indeed, the relations \eqref{eq.min-Lov-ext} and then \eqref{eq.chain}, \eqref{eq.min-weber}
imply, for any $S\in {\cal C}_{\pi}$,
$$
\lov (\incS_{S}) \stackrel{\eqref{eq.min-Lov-ext}}{\leq}
\sum_{j \in N} x^{\pi}_{j}\cdot\incS_{S}(j)=\sum_{i\in S}\, x^{\pi}_{i} \stackrel{\eqref{eq.chain}}{=} m(S)
\stackrel{\eqref{eq.min-weber}}{=} \min_{\tau\in\Upsilon}\, \sum_{i\in S}\, x^{\tau}_{i}=
\min_{\tau\in\Upsilon}\,\sum_{j\in N} x^{\tau}_{j}\cdot \incS_{S}(j)
\stackrel{\eqref{eq.min-Lov-ext}}{=}\lov (\incS_{S}).
$$
Hence, the first inequality must be the equality, which allows one to conclude that $\lov$ coincides with the linear function
$y\in\nabla_{\pi} \mapsto \sum_{j\in N}x^{\pi}_{j}\cdot y_{j}$.
\smallskip

On the other hand, the form of canonical max-representation \eqref{eq.max-repr} of $m\in\supmogame$
leads to introducing another extension $\ros$ of $m$, namely
\begin{equation}
\ros (y) := \max_{\tau\in\Omega}~ (- z^{\tau}_{\emptyset}+ \sum_{j\in N} z^{\tau}_{j}\cdot y_{j} )
\qquad
\mbox{for $y\in [0,1]^{N}$, where $z^{\tau}:=[z^{\tau}_{i}]_{i\in \{\emptyset\}\cup N}$.}
\label{eq.max-RW-ext}
\end{equation}
This extension, inspired by the max-representation by Rosenm\"{u}ller and Weidner \cite{RW74},
is the pointwise maximum of affine functions, and, therefore, a {\em convex function\/} on $[0,1]^{N}$.
The specialty of this extension $\ros$ is that, for any $S\subseteq N$,
it is affine on the simplex
$$
\nabla_{S} := \conv \{\, \incS_{T}\, :\ T\in \mbox{tuft}\,(S)\,\} \subseteq [0,1]^{N}\,.
$$
Indeed, this conclusion can be derived (by an analogous consideration as in the case of $\lov$) from the condition (ii)
in the definition of the canonical max-representation and \eqref{eq.max-repr}.
\smallskip

Let us compare both extensions. First, $\lov$ and $\ros$ coincide on the edges of
the hypercube $[0,1]^{N}$, which are the segments connecting $\incS_{S}$ and $\incS_{S\setminus\{i\}}$ for $i\in S\subseteq N$.
One can also show that $\lov (y)\geq\ros (y)$ for any $y\in [0,1]^{N}$.
The concave extension $\lov$ is affine on every $\nabla_{\pi}$, $\pi\in\Upsilon$
and these are full-dimensional simplices covering $[0,1]^{N}$. On the other hand,
the convex extension $\ros$ is ensured to be affine on other simplices $\nabla_{S}$, $S\subseteq N$,
which are lower-dimensional in general.
These do not cover $[0,1]^{N}$ and $\nabla_{N}$ is the only one of them that is full-dimensional.
To illustrate the difference note that, for distinct $a,b\in N$ and $Z\subseteq N\setminus\{ a,b\}$,
the values in $y=\frac{1}{2}\cdot \incS_{\{ a\}\cup Z}+\frac{1}{2}\cdot \incS_{\{ b\}\cup Z}=
\frac{1}{2}\cdot \incS_{\{ a,b\}\cup Z}+\frac{1}{2}\cdot \incS_{Z}$ are
$\lov (y)= \frac{1}{2}\cdot m(\{ a,b\}\cup Z)+\frac{1}{2}\cdot m(Z)$ and
$\ros (y)=\frac{1}{2}\cdot m(\{ a\}\cup Z)+\frac{1}{2}\cdot m(\{ b\}\cup Z)$.

The subtle argument why $\lov\,$ is more natural extension than $\ros$ is as follows.
The maximal domains in $[0,1]^{N}$ on which $\lov\,$ is affine are the simplices $\nabla_{\pi}$, $\pi\in\Upsilon$.
In particular, the vertices of these maximal linearity domains for $\lov$ are the vertices of the hypercube.
This is the case no matter what is the extended game $m\in\supmogame$.
However, the maximal domains in $[0,1]^{N}$ on which $\ros$ is affine may vary, they
could have vertices with fractional coordinates.
On the top of that, the vertices of the maximal affine domains for $\ros$ do
depend on the extended game $m$.

\begin{table}
\[
\begin{array}{cccccccccc}
~ & [\,x_{a}, & x_{b}, & x_{c} \,] &~~& \emptyset & \{a,b\} & \{a,c\} &\{b,c\} & N\\[0.3ex] \hline
  & [\, 1, & 1, & 1\,] && &&&&\bullet  \\
  & [\, 1, & 1, & 0\,] && &\bullet &&&\bullet  \\
  & [\, 1, & 0, & 1\,] && &&\bullet &&\bullet  \\
  & [\, 0, & 1, & 1\,] & & &&&\bullet & \bullet \\
  & [\, 1, & 0, & 0\,] && \bullet &\bullet & \bullet && \\
  & [\, 0, & 1, & 0\,] && \bullet &\bullet &&\bullet &  \\
  & [\, 0, & 0, & 1\,] &&  \bullet &&\bullet &\bullet & \\
  & [\, 0, & 0, & 0\,] &&  \bullet &&&& \\[0.4ex]
  & [\, \frac{1}{2}, & \frac{1}{2}, & \frac{1}{2}\,] &&  \bullet &\bullet &\bullet &\bullet & \bullet \\[0.4ex]
  & [\, 1, & \frac{1}{3+2\cdot\lambda}, & \frac{1+\lambda}{3+2\cdot\lambda}\,] &&  &\bullet &\bullet && \bullet \\[0.4ex]
  & [\,  \frac{1}{3+2\cdot\lambda}, & 1, & \frac{1+\lambda}{3+2\cdot\lambda}\,] &&  &\bullet &&\bullet & \bullet \\[0.4ex]
  & [\, \frac{1+\lambda}{3+2\cdot\lambda} & \frac{1+\lambda}{3+2\cdot\lambda} & 1\,] && &&\bullet &\bullet & \bullet \\[0.4ex] 
\end{array}
\]
\caption{Vertices of tightness domains in the max-representation from Example \ref{exa.RW-exten}.}\label{tab.1}
\end{table}

\begin{example}\label{exa.RW-exten}
Assume $N=\{a,b,c\}$ and consider a parameterized class of games over $N$ ~
$$
m_{\lambda}=(3+\lambda)\cdot\delta_{N} + (1+\lambda)\cdot\delta_{\{a,b\}} + \delta_{\{a,c\}} + \delta_{\{b,c\}},
\qquad \mbox{where $\lambda\geq 0$.}
$$
Then, no matter what $\lambda$ is, the canonical max-representation $\ros$ of $m_{\lambda}$
consists of five modular functions, namely $l^{T}$ given by \eqref{eq.RW-formula} for\,
$T\in {\cal T} = \{\, \emptyset , \{a,b\}, \{ a,c\}, \{b,c\}, N\,\}$.

Thus, $[0,1]^{N}$ splits into five tightness domains, namely
$\{ y\in [0,1]^{N}\, :\ l^{T}(y)=\ros (y)\}$ for $T\in {\cal T}$.
These polytopes are determined by twelve points in the hypercube given by Table \ref{tab.1}
in which the bullets indicate to which tightness domains they belong.
For example, the tightness domain for $T=\{ a,b\}$ has six vertices, besides three vertices of the
two-dimensional simplex $\nabla_{\{ a,b\}}$ it has three other vertices, namely
$$[{1}/{2},{1}/{2},{1}/{2}],~
[1, {1}/{(3+2\cdot\lambda)},{(1+\lambda)}/{(3+2\cdot\lambda)}], ~[{1}/{(3+2\cdot\lambda)}, 1, {(1+\lambda)}/{(3+2\cdot\lambda)}]\,.
$$
Observe that the domains are different for different parameters $\lambda\geq 0$, that is, for different represented games $m_{\lambda}$.
\end{example}

\begin{remark}\rm\label{rem.RW-misuse}
The motivation for the max-representation of supermodular games given in \cite[p.\,245]{RW74}
is a~little bit strange. Rosenm\"{u}ller and Weidner seem to misuse the terminology,
specifically, the facts that supermodular games are named ``convex" in game theory
and that the functional form of a modular set function is analogous to
that of an affine point function, for which reason they re-name modular functions to ``affine".
They seem to segue from the fact the any convex function (of a real variable) is the maximum
of affine functions (of a real variable) to the idea of represent any supermodular set function
as the maximum of a collection of modular set functions. However, the reasons why a supermodular
game is named ``convex" are completely different. The main reason is recalled in Remark \ref{rem:coopconvex}.
Another minor reason is that an important special case of a~supermodular game is the so-called
{\em convex measure game} treated in \S\,\ref{ssec.conv-measure-game}, defined as the composition of
a real convex function with a non-negative additive set function.
\end{remark}

Another supportive argument in favor of the min-representation is that the canonical max-representation
introduced by Rosenm\"{u}ller and Weidner \cite{RW74} does not have such an elegant geometric interpretation
as the core-based min-representation has. Indeed, the vertices of the core $\cor (m)$ for $m\in\supmogame$
coincide with the vertices of its extended version
$$
E(m) := \left\{\, [v_{i}]_{i\in N} : ~
\forall\, S\subseteq N ~\sum_{i\in S} v_{i}\geq m(S)\,\right\},
$$
which naturally corresponds to the min-representation of $m$ by means of additive functions.
Note that one can show that $\cor (m)$ is the Pareto minimum of $E(m)$; compare with Theorem~2.3 from \cite{fuj91}.
However, the vertices of the following polyhedron
$$
R(m) := \left\{ z\in {\dv R}^{\{\emptyset\}\cup N}\, : ~ z_{\emptyset}\geq 0 , z_{i}\geq 0,~~ i\in N
 ~~\& ~~ -z_{\emptyset}+\sum_{i\in S} z_{i}\leq m(S),~~ S\subseteq N\,\right\},
$$
which naturally corresponds to the max-representation of $m$ by means of modular functions of the considered type,
need not have the form of the canonical max-representation.

\begin{example}\label{exa.RW-core}
Assume $N=\{a,b,c,d\}$ and consider the game $m_{\dag}$ from Example \ref{exa.Vamosi}.
Then the above polyhedron $R(m_{\dag})$ has 13 vertices but only 11 of them correspond to sets $T\subseteq N$ in the sense
\eqref{eq.RW-formula}. The remaining two vertices $[z_{\emptyset}\,|\,z_{a},z_{b},z_{c},z_{d}]$ of $R(m_{\dag})$ are $[2\,|\,2,1,1,1]$ and $[2\,|\,1,2,1,1]$.
The tightness set class for the modular function given by
$[z_{\emptyset}\,|\,z_{a},z_{b},z_{c},z_{d}]=[2\,|\,2,1,1,1]$ is the class
$$
{\cal Q} = \{\,  \{ a\}, \{ a,b\}, \{ a,c\}, \{ a,d\}, \{ c,d\}, \{ a,b,c\}, \{ a,b,d\}, \{ a,c,d\}\, \},
$$
which strictly contains the tightness set class for the function $l^{\{ a,c,d\}}$ from the canonical max-representation of $m_{\dag}$.
Thus, $R(m_{\dag})$ leads to a tighter max-representation than the canonical max-representation.
\end{example}

The facts mentioned above lead us to the opinion that, for a supermodular game $m$, the concave extension proposed
by Lov\'{a}sz \cite{lov83} is more natural than the convex extension \eqref{eq.max-RW-ext}
inspired by Rosenm\"{u}ller and Weidner \cite{RW74}, although they both
may appear to be convenient. For this reason we prefer the min-representation to the max-representation.

\section{Remarks on other extremality criteria}\label{sec.remark-other}

In the 1970's two other papers were published which provide very simple
criteria to recognize extreme supermodular functions in two special cases.
These criteria basically consist in testing whether a collection of linear
constraints on a vector in ${\dv R}^{N}$ has a unique solution.

The 1973 paper by Rosenm\"{u}ller and Weidner \cite{RW73}, which
had probably been the source of inspiration for their later 1974 general criterion \cite{RW74},
provides such a criterion for {\em convex measure games}, that is, for games expressible
as compositions of a convex function with a non-negative additive set function.

The 1978 paper by Nguyen \cite{ngu78} deals with non-decreasing submodular games
and gives an analogous criterion for extremality of such games in case they correspond to
a {\em matroid} \cite{oxl92}; of course, his criterion can be ``converted" to the supermodular
case. Moreover, Nguyen also presented in \cite{ngu78} a kind of extension to the case of
a general non-decreasing submodular game based on (his) concept of a matroidal {\em expansion}.
Nevertheless, that extension of his does not seem to lead to a practical criterion to
test extremality of such general games.

The aim of this section is to recall those special criteria and illustrate, by means of examples,
that they differ from each other and from our new criterion as well.

\subsection{The case of convex measure games}\label{ssec.conv-measure-game}
The paper \cite{RW73} deals with non-negative supermodular games $m$ satisfying $m(N)=1$. The motivational
task is when a convex measure game $m$ is an extreme point of this polytope. Some starting intuitive consideration
leads the authors to the restriction to a particular form of convex measure games.
Specifically, provided one omits the trivial case of modular $m$, the game is assumed to be a
composition $m=f^{\alpha}\circ\mu$ where $\mu :\caP\rightarrow [0,1]$ is an additive set function with
$\mu (N)=1$ and $f^{\alpha}:[0,1]\rightarrow [0,1]$ a~convex function of the form
$$
f^{\alpha}(t) = \frac{1}{1-\alpha}\cdot\max\,\{ t-\alpha , 0\}\quad
\mbox{for~} t\in [0,1], \quad \mbox{determined by a parameter $0<\alpha < 1$.}
$$
Rosenm\"{u}ller and Weidner \cite{RW73} show that such a representation of  $m$
in terms of $\mu$ and $\alpha$ is unique under an~additional requirement
$\mu (\{ i\})\leq 1-\alpha$ for $i\in N$.

The necessary and sufficient condition for a convex measure game $m$ to be extreme in terms of
such a~``canonical" representation is as follows. One introduces the class of sets
$$
{\cal T} := \{\, T\subseteq N\,:\ \mu (T)=\alpha\}
$$
and the result is that $m$ is extreme iff the system of linear constraints
\begin{equation}
\forall\,T\in {\cal T}\qquad  \sum_{i\in T} x_{i} = 0
\label{eq.RW73}
\end{equation}
on real numbers $x_{i}$, $i\in N$ has a unique solution, namely $\bar{x}_{i}=0$ for any $i\in N$.

Note that the non-degeneracy condition from \cite{RW74} mentioned in \S\,\ref{ssec.Rosen-recall}
can be interpreted as an indirect generalization of this simple condition. Indeed, the above
representation of a (standardized) convex measure game $m=f^{\alpha}\circ\mu$ means it has a max-representation
consisting of two modular functions only, namely $l^{N}$ and $l^{\emptyset}\equiv 0$ given by \eqref{eq.RW-formula}.
The class ${\cal T}$ is then nothing but the tightness set class ${\cal Q}^{N}_{0}\equiv {\cal Q}^{N}\cap {\cal Q}^{\emptyset}$.
For this reason, despite that the canonical max-representation of $m$ mentioned in \S\,\ref{ssec.Rosen-recall} may
involve additional functions, the constraints (x)-(z) from \S\,\ref{ssec.Rosen-recall} have a unique solution up to a multiple iff
\eqref{eq.RW73} has solely the zero solution. We leave the verification of this statement as a simple exercise to the reader.
The above statements are illustrated by the following example.

\begin{example}\label{exa.RW-73-74}
Assume $N=\{a,b,c,d\}$ and consider the game $m$ over $N$ given by
$$
m = \delta_{N} + \frac{1}{2}\cdot\delta_{\{a,b,c\}} + \frac{1}{2}\cdot\delta_{\{a,b,d\}} + \frac{1}{2}\cdot\delta_{\{a,c,d\}}\,.
$$
It is a convex measure game represented by
$\mu (a)={2}/{5}$, $\mu (b)=\mu (c)=\mu (d)={1}/{5}$ and $\alpha={3}/{5}$.
In particular, the class ${\cal T}$ consists of four sets, namely $\{ a,b\}$, $\{ a,c\}$, $\{ a,d\}$ and $\{ b,c,d\}$.
Clearly, the only solution to \eqref{eq.RW73}, that is, of
$$
x_{a}+x_{b} = x_{a}+x_{c} = x_{a}+x_{d} = x_{b}+x_{c}+x_{d} = 0
$$
is the zero vector in this case.
On the other hand, the canonical max-representation given by \eqref{eq.RW-formula} consists of five modular functions.
Their coefficients and tightness set classes are given in Table \ref{tab.2}.
Since ${\cal Q}^{N}\cup {\cal Q}^{\emptyset}=\caP$, $m$ has a max-representation consisting solely of
$l^{N}$ and $l^{\emptyset}\equiv 0$. Moreover, ${\cal Q}^{N}_{0}={\cal Q}^{N}\cap {\cal Q}^{\emptyset}={\cal T}$.
Hence, the equations (z) from \S\,\ref{ssec.Rosen-recall} for $y_{i}:= y(N,i)$, $i\in \{\emptyset\}\cup N$ lead to
$\sum_{i\in S} y_{i}=y_{\emptyset}$ for any $S\in {\cal T}$, which then reduces to
$\sum_{i\in S} y_{i}=\sum_{j\in T} y_{j}$ for any $S,T\in {\cal T}$. We also know
that this reduced system has a positive solution, for example, $y_{i}=\mu (i)$ for $i\in N$.
A simple consideration leads to a conclusion that the reduced system has a unique solution up to a multiple iff
the only solution of \eqref{eq.RW73} is $\bar{x}_{i}=0$ for $i\in N$.

\begin{table}
\[
\bbordermatrix{
 ~~~T\subseteq N  & \emptyset & \VR a & b & c & d \cr
 \{ a,b,c,d\}~~ & \frac{3}{2} & \VR 1 & \frac{1}{2} & \frac{1}{2} & \frac{1}{2} \cr
 ~\{a,b,c\} & 1 & \VR \frac{1}{2} & \frac{1}{2} & \frac{1}{2} & 0 \cr
 ~\{a,b,d\} & 1 & \VR \frac{1}{2} & \frac{1}{2} &Â 0 & \frac{1}{2} \cr
 ~\{a,c,d\}  & 1 & \VR \frac{1}{2} & 0 & \frac{1}{2} & \frac{1}{2} \cr
 ~~~~~\emptyset & 0 & \VR 0 & 0 &Â 0 & 0 \cr}
 \qquad
 \begin{array}{l}
{\cal Q}^{T}\\
 ab, ac, ad, abc, abd, acd, bcd, abcd\\
 ab, ac, bc, abc, bcd\\
 ab, ad, bd, abd, bcd\\
 ac, ad, cd, acd, bcd\\
 \emptyset , a, b, c, d, ab, ac, ad, bc, bd, cd, bcd\\
 ~~\\
 \end{array}
\]
\caption{The max-representation and tightness set classes in Example \ref{exa.RW-73-74}.}\label{tab.2}
\end{table}

Note that there are seven vertices of the core of $m$, namely $[x_{a},x_{b},x_{c},x_{d}]$ of the form
$$
[1,0,0,0],\ [{1}/{2},{1}/{2},0,0],\ [{1}/{2},0, {1}/{2},0],\ [{1}/{2},0,0,{1}/{2}],\
 [0,{1}/{2},{1}/{2},0],\ [0,{1}/{2},0, {1}/{2}],\ [0,0,{1}/{2},{1}/{2}]\,.
$$
Thus, the system of linear equations (a)-(b) from \S\,\ref{ssec.main-formul} involves many more variables
than the special system \eqref{eq.RW73}. One can certainly expect this to happen because the criterion from \cite{RW73}
was particularly tailored for the case of convex measure games.
\end{example}

On the other hand, if one does try to apply the criterion based on \eqref{eq.RW73} to some given
game over $N$, say to the normalized version of the game $t^{\star}$ from Example \ref{exa.negative},
then one spends quite a lot of time to compute the respective ``canonical" representation.
In the case of $t^{\star}/22$ one finally gets $\alpha ={1}/{2}$, $\mu (a)={5}/{22}$, $\mu (b)={4}/{11}$,
and $\mu (c)={9}/{22}$, which results in the empty class ${\cal T}$. Thus, the application of the criterion
based on \eqref{eq.RW73} itself is trivial, but a~tedious task may be to get the required ``canonical" representation.

\subsection{The case of matroids}\label{ssec.matroid}
The paper \cite{ngu78} deals with the cone of non-decreasing submodular games. A special case of such
a game is the {\em rank function\/} of a {\em matroid}, which is an integer-valued non-decreasing submodular game $r$
satisfying $r(S)\leq |S|$ for any $S\subseteq N$. Note that an apparently weaker but equivalent formulation of
the latter condition is $r(\{ i\})\leq 1$ for any $i\in N$.

Theorem 2.1.5 of \cite{ngu78} gives a necessary and sufficient condition for a rank function $r$ of a matroid to generate
an extreme ray of the above cone. To formulate that result in a suitable way we need the next concept.

\begin{definition}\label{def.support}\rm
The {\em support\/} of a game $r\in {\dv R}^{N}$ is the least set $M\subseteq N$ such that
$$
\forall\, S\subseteq N\qquad r(S)=r(S\cap M)\,.
$$
\end{definition}

The appropriate formulation of the condition from \cite{ngu78} is that the corresponding matroid
restricted to the support of $r$ is {\em connected}, which is a well-known concept in matroid theory;
see \cite[chapter\,4]{oxl92}. Nevertheless, the condition has an alternative formulation in terms of linear
constraints on a vector in ${\dv R}^{M}$, where $M$ is the support of $r$.
Specifically, the rank function $r$ defines the class of matroidal {\em bases} (compare \cite[\S\,1.3]{oxl92}):
$$
{\cal B} := \{\, B\subseteq N\, :\ B \mbox{ maximal such that~} r(B)=|B|\,\}\equiv
\{\, B\subseteq N\, :\ r(B)=|B|=r(N)\,\}.
$$
It makes no problem to show that the support $M$ of $r$ coincides with the union of bases.
Provided $M\equiv\bigcup {\cal B}\neq\emptyset$, the condition is that the system of linear constraints
\begin{equation}
\forall\, B,C\in {\cal B}\qquad  \sum_{i\in B}\, y_{i}=\sum_{j\in C}\, y_{j}
\label{eq.nguen}
\end{equation}
on real numbers $y_{i}$, $i\in M$ has a unique solution up to a multiple.
Since all the bases of a~matroid have the same cardinality $r(N)$,
\eqref{eq.nguen} always has a constant solution $\bar{y}_{i}=u\in {\dv R}$, $i\in M$, and
the condition can equivalently be stated that the linear equation system
$\sum_{i\in B} x_{i}=0$ for $B\in {\cal B}$  has solely the zero solution
$\bar{x}_{i}=0$ for $i\in M$.

\begin{remark}\label{rem.Nguyen-error}\rm
Theorem 2.1.5 in \cite{ngu78} was formulated in a slightly misleading way. In fact, it says that
``{\sl $r$ is extreme in the respective cone iff the matroid is connected}".
This is not true as stated. Here is a simple counter-example: consider $N=\{a,b,c\}$ and put
$$
r = \delta_{\{a,b,c\}} + \delta_{\{a,b\}} + \delta_{\{a,c\}} + \delta_{\{b,c\}}
+ \delta_{\{a \}} + \delta_{\{ b\}},
$$
which is the rank function of a matroid over $N$ with bases $\{a \}$ and $\{ b\}$. The corresponding
matroid is {\em not\/} connected despite that $r$ {\em does} generate an extreme ray of the cone.
The reason why the matroid is not connected is that $\{ a,b\}$ and $\{c\}$ are non-trivial separators; see \cite[\S\,4.2]{oxl92}
for related concepts.

The point is that an additional technical assumption $r(\{ i\})=1$ for any $i\in N$
is tacitly used despite it is omitted in the formulation of \cite[Theorem~2.1.5]{ngu78}.
The assumption means that the corresponding matroid has no {\em loops}, see \cite[p.\,13]{oxl92}
for the related concept. It is indeed applied in the proof, specifically on page 378 of \cite{ngu78},
the implication (iii)$\Rightarrow$(i).
This tacit assumption is stated in \cite{ngu78} earlier in the text as a convention,
which is, unfortunately, hidden more than one page before the very formulation of Theorem 2.1.5.
However, in our paper, we have chosen to re-formulate Nguyen's result in such a way that the above-mentioned technical assumption is avoided.
\end{remark}

To put that result into our context realize that a submodular game $r$ over $N$ is non-decreasing iff $r(N)\geq r(N\setminus \{i\})$ for $i\in N$.
Thus, one can always write $r$ as the sum of a (non-negative) modular game and a submodular game $\bar{r}$ satisfying
\begin{equation}
\bar{r}(N)=\bar{r}(N\setminus \{i\})\qquad \mbox{for any $i\in N$}.
\label{eq.u-stand}
\end{equation}
Of course, the modular game is the linear combination of $m^{\uparrow i}$, $i\in N$
with non-negative coefficients $r(N)-r(N\setminus\{ i\})$, $i\in N$.
In fact, such a decomposition of $r$ is uniquely determined: to this end introduce the notation
\begin{eqnarray*}
\komstagame && \mbox{for the linear space of games $\bar{r}$ over $N$ satisfying \eqref{eq.u-stand},}\\
\submogame && \mbox{for the cone of submodular games $\bar{r}$ over $N$ satisfying \eqref{eq.u-stand},}
\end{eqnarray*}
and realize that the only modular game in $\komstagame$ is the zero function.
In particular, the above result from \cite{ngu78} essentially gives a criterion to recognize whether a rank function $\bar{r}$
of a matroid satisfying \eqref{eq.u-stand} generates an extreme ray of $\submogame$.

One can transform $\komstagame$ by an invertible linear mapping onto $\stagame$ which transforms
$\submogame$ onto $\supmogame$. In fact, there are two such  suitable transformations,
which are complementary to each other; we discuss this complementarity topic later in Remark \ref{rem.maps-complement}.
Since an invertible linear mapping transforms extreme rays to extreme rays, this gives us implicitly a criterion to recognize
some of the extreme rays in $\supmogame$.
\smallskip

From the point view of {\em conditional independence\/} interpretation of these games
(see \ref{sec.apex-CI}, Remark \ref{rem.CI-interpret}) the correspondence
$m\in\stagame\longleftrightarrow\bar{r}\in\komstagame$ given by
\begin{equation}
\begin{array}{rcl}
m(S) &=& -\bar{r}(S) + \sum_{i\in S}\, \bar{r}(\{ i\})\quad \mbox{for $S\subseteq N$},\\[0.8ex]
\bar{r}(T) &=& -m(T) + |T|\cdot m(N) -\sum_{i\in T}\, m(N\setminus \{ i\})\quad \mbox{for $T\subseteq N$},
\end{array}
\label{eq.r-to-m}
\end{equation}
seems to be natural because it has the property
$$
\forall\, A,B\subseteq N\quad
m(A\cup B)+m(A\cap B)-m(A)-m(B) = -\bar{r}(A\cup B)-\bar{r}(A\cap B)+\bar{r}(A)+\bar{r}(B)\,,
$$
for which reason the conditional independence structures given by $m$ and $\bar{r}$ coincide.
Note that \eqref{eq.r-to-m} is, in fact, the multiplication by $(-1)$ adapted to fit into the spaces
$\stagame$ and $\komstagame$. Other authors  prefer the correspondence
$m\in\stagame\longleftrightarrow \bar{r}^{*}\in\komstagame$ given by a simpler formula
\begin{equation}
\begin{array}{rcl}
m(S) &=& \bar{r}^{*}(N)-\bar{r}^{*}(N\setminus S) \quad \mbox{for $S\subseteq N$},\\[0.8ex]
\bar{r}^{*}(T) &=& m(N) -m(N\setminus T) \quad \mbox{for $T\subseteq N$},
\end{array}
\label{eq.m-to-dual-r}
\end{equation}
see the concept of a {\em dual submodular system} from \cite[p.\ 37]{fuj91}.
However, this correspondence does not preserve the conditional independence interpretation.

Clearly, the rank function $\bar{r}$ of a matroid satisfying \eqref{eq.u-stand} is transformed by \eqref{eq.r-to-m} to an integer-valued
standardized supermodular game $m$ satisfying
$$
m(N)- m(N\setminus\{ i\})\equiv\bar{r}(\{ i\})\leq 1\qquad
\mbox{for any $i\in N$}.
$$
The mapping \eqref{eq.m-to-dual-r} also transforms rank functions $\bar{r}^{*}$ satisfying \eqref{eq.u-stand}
to the same class of games. This is the class of games $m\in\supmogame$ to which the matroidal criterion is applicable.
An alternative characterization of this class of games is that
all vertices of the core $\cor (m)$ are zero-one vectors; see Corollary \ref{cor.zero-one-core} in \ref{sec.apex-supermod}.
Note that the case of convex measure games is not covered by the matroid case because
the integer-valued version $\tilde{m}$ of the game from Example \ref{exa.RW-73-74}
satisfies $\tilde{m}(N)-\tilde{m}(\{b,c,d\})=2$.

Given $m\in\supmogame$ with zero-one $\ext (\cor (m))$, one can apply the formula \eqref{eq.r-to-m}
to get the respective rank function $\bar{r}$ and determine ${\cal B}$ on basis of it.
Alternatively, the transformation \eqref{eq.m-to-dual-r} can be used instead.
We illustrate the procedure in the next example, which also shows that
the case of matroids is not covered by the convex measure game case.

\begin{example}\label{exa.nguy78}
Put $N=\{a,b,c,d\}$ and consider an integer-valued supermodular game
$$
m = 2\cdot\delta_{N} + \delta_{\{a,b,c\}} + \delta_{\{a,b,d\}} + \delta_{\{a,c,d\}} + \delta_{\{b,c,d\}} + \delta_{\{a,b\}}.
$$
One has $m(N)- m(N\setminus\{ i\})=1$ for any $i\in N$ and the corresponding rank function is
$$
\bar{r} = 2\cdot\delta_{N} + 2\cdot\sum_{i\in N} \delta_{N\setminus\{ i\}}
+ 2\cdot\sum_{\substack{S\subseteq N,\, |S|=2,\\ S\neq \{a,b\}}} \delta_{S} +\delta_{\{a,b\}}
+\sum_{i\in N} \delta_{\{ i\}},
$$
which means one has
$$
{\cal B} = \{\, S\subseteq N :\ |S|=2,\, S\neq \{a,b\}\,\} = \{\, \{a,c\}, \{a,d\}, \{b,c\}, \{b,d\} , \{c,d\}\,\}\,.
$$
Clearly, every solution to \eqref{eq.nguen} is constant in this case, which implies that $m$ is an extreme supermodular game
by Nguyen's criterion. Alternatively, the mapping \eqref{eq.m-to-dual-r} gives
$$
\bar{r}^{*} = 2\cdot\delta_{N} + 2\cdot\sum_{i\in N} \delta_{N\setminus\{ i\}}
+ 2\cdot\sum_{\substack{S\subseteq N,\, |S|=2,\\ S\neq \{c,d\}}} \delta_{S} +\delta_{\{c,d\}}
+\sum_{i\in N} \delta_{\{ i\}},
$$
which defines another class of bases, namely
$$
{\cal B}^{*} =  \{\, \{a,b\}, \{a,c\}, \{a,d\}, \{b,c\}, \{b,d\}\,\}\,.
$$

On the other hand, $m$ is not a convex measure game.
Indeed, assume for a contradiction $m=f\circ \mu$, where $\mu$ is a probability measure and
$f:[0,1]\to [0,2]$ convex with $f(0)=0$ and $f(1)=2$. Then, since $f$ cannot be constant on any sub-interval of
$f_{-1}(0,2]$, one has
$$ m(A)=f(\mu (A))=f(\mu (B))=m(B)>0 ~\Rightarrow~ \mu (A)=\mu (B)\quad
\mbox{for any $A,B\subseteq N$}.
$$
Thus, $\{a,b\}\subseteq \{a,b,c\}$ and $m(\{a,b\})=m(\{a,b,c\})>0$ implies $\mu (\{c\})=0$, which
gives $\mu (\{a,b,d\})=\mu (N)~\Rightarrow~ 1=m(\{a,b,d\})=m(N)=2$, a contradiction.

\begin{table}
$$
\bbordermatrix{
 ~~T\subseteq N  & \emptyset & \VR a & b & c & d \cr
 \{ a,b,c,d\}~~ & 2 & \VR 1 & 1 & 1 & 1 \cr
 ~~\{a,b\} & 1 & \VR 1 & 1 & 0 & 0 \cr
 ~~~~~\emptyset & 0 & \VR 0 & 0 &Â 0 & 0 \cr}
 \qquad\quad
 \begin{array}{l}
 {\cal Q}^{T}\\
 ac, ad, bc, bd, cd, abc, abd, acd, bcd, abcd\\
 a, b, ab, ac, ad, bc, bd, abc, abd\\
 \emptyset , a, b, c, d, ac, ad, bc, bd, cd\\
 ~~
\end{array}
$$
\caption{The max-representation and tightness set classes in Example \ref{exa.nguy78}.}\label{tab.3}
\end{table}

The canonical max-representation of $m$ from \S\,\ref{ssec.Rosen-recall} consists of three modular functions,
shown in Table~\ref{tab.3}, while there are five vertices of core $\cor (m)$,
namely the vectors $[x_{a},x_{b},x_{c},x_{d}]$ of the form
$$
[1,1,0,0],\ [1,0,1,0],\ [1,0,0,1],\ [0,1,1,0],\ [0,1,0,1]\,,~~
$$
which are the incidence vectors of sets in ${\cal B}^{*}$.
Thus, both the system of linear equation (x)-(z) from  \S\,\ref{ssec.Rosen-recall} and the system
(a)-(b) from \S\,\ref{ssec.main-formul} involve more variables than Nguyen's matroid-based criterion.
\end{example}

\begin{remark}\label{rem.maps-complement}\rm
This is to explain the relation of linear mappings given by \eqref{eq.r-to-m} and \eqref{eq.m-to-dual-r}.
Given $m\in\stagame$, the corresponding games $\bar{r}$ and $\bar{r}^{*}$ are in the following
relation
\begin{equation}
\bar{r}^{*}(T) = \bar{r}(N\setminus T) -\bar{r}(N) +\sum_{i\in T} \bar{r}(\{i\})\qquad  \mbox{for $T\subseteq N$},
\label{eq.r-dual}
\end{equation}
which defines an invertible linear mapping of $\komstagame$ onto itself. Since the inverse of
\eqref{eq.r-dual} on $\komstagame$ is itself, it can be viewed as a kind of duality transformation on $\komstagame$.
Moreover, it maps $\submogame$ onto itself and rank functions of matroids within $\komstagame$
to rank functions. Since $\bar{r}^{*}(\{i\})=\bar{r}(\{i\})$ for $i\in N$ the supports of rank functions $\bar{r}$ and
$\bar{r}^{*}$ coincide. The relation of matroids corresponding to $\bar{r}$ and
$\bar{r}^{*}$ is that their restrictions to the (shared) support $M\subseteq N$ are {\em dual matroids};
see \cite[chapter 2]{oxl92} for this concept. That means,
if $M\neq\emptyset$, the class ${\cal B}^{*}$ of bases given by $\bar{r}^{*}$ is
$$
{\cal B}^{*} ~=~ \{\, M\setminus B\,:\ B\in {\cal B}\,\},\qquad
\mbox{where ${\cal B}$ is the class of bases given by $\bar{r}$.}
$$
Since the condition \eqref{eq.nguen} for ${\cal B}$ can be re-written as
$$
\forall\, B,C\in {\cal B}\qquad  \sum_{i\in B\setminus C}\, y_{i}=\sum_{j\in C\setminus B}\, y_{j}\,,
$$
it clearly gives the same requirement as \eqref{eq.nguen} for ${\cal B}^{*}$. In particular, no matter whether
one decides to apply  Nguyen's criterion either to $\bar{r}$ or to $\bar{r}^{*}$ one gets the same
condition.

Analogously, given $\bar{r}\in\komstagame$ the corresponding games $m$ and $m^{*}$ in $\stagame$ determined
by \eqref{eq.r-to-m} and \eqref{eq.m-to-dual-r} are related by a duality relation on  $\stagame$ given by
$$
m^{*}(S) ~=~ m(N\setminus S) + (|S|-1)\cdot m(N) -\sum_{i\in S}\, m(N\setminus\{ i\})\qquad
\mbox{for $S\subseteq N$}.
$$
Of course, the mappings $m\mapsto m^{*}$ transforms $\supmogame$ onto itself. These are the reasons why we consider
\eqref{eq.r-to-m} and \eqref{eq.m-to-dual-r} to be complementary to each other.
\end{remark}

There is even closer relation of matroidal bases and the (vertices of the) respective core,
which allows us to re-formulate Nguyen's matroidal criterion as follows.

\begin{coro}\label{cor.Nguyen-trans}\rm
Let $m\in\supmogame$ be such that ${\cal X}=\ext(\cor (m))$ consists of zero-one vectors and
the support of $m$, denoted by $M$, is non-empty.
Then $m$ generates an extreme ray of $\supmogame$ iff every solution to the system of linear constraints on $y\in {\dv R}^{M}$ ~~
\begin{equation}
\forall\, v,w\in {\cal X}=\ext(\cor (m))\qquad 0=\sum_{i\in M}\, (v_{i}-w_{i})\cdot y_{i}
\label{eq.Nguyen-trans}
\end{equation}
is constant, that is, $\bar{y}_{i}=u$ for some $u\in {\dv R}$.
\end{coro}

\begin{proof}
By Corollary \ref{cor.zero-one-core}, $m$ is integer-valued and $m(N)-m(N\setminus\{ i\})\leq 1$ for any $i\in N$.
Re-write the definition of $\cor (m)$ in terms of the respective rank function $\bar{r}^{*}$
given by \eqref{eq.m-to-dual-r}:
$$
\cor (m) = \{\, [v_{i}]_{i\in N}\in {\dv R}^{N}\, : ~ \sum_{i\in N} v_{i}= \bar{r}^{*}(N) \,~\&~\,
\forall\, T\subseteq N ~\sum_{i\in T} v_{i}\leq \bar{r}^{*}(T)\,\}
$$
and apply Proposition 3.12(iii) in
\cite{wol98} to deduce that the vertices of $\cor (m)$ are just the incidence vectors of the bases
of the matroid given by $\bar{r}^{*}$. Note that the same observation was made in \cite[Proposition 2.5]{ABD11}
and \cite[Proposition 2.2.5]{dok11}.
Thus, \eqref{eq.nguen} turns into \eqref{eq.Nguyen-trans}.
\end{proof}

\begin{remark}\label{rem.Nguyen-trans}\rm
The above condition \eqref{eq.Nguyen-trans} cannot be
extended to an extremality criterion for general $m\in\supmogame$.
Put, for example, $N=\{ a,b,c,d\}$ and  $m=2\cdot\delta_{N} + \delta_{\{a,b,c\}} + \delta_{\{b,c,d\}}$.
Then the core $\cor (m)$ has seven vertices,  namely $[x_{a},x_{b},x_{c},x_{d}]$ of the form
$$
[1,1,0,0],\ [1,0,1,0],\ [1,0,0,1],\ [0,1,0,1],\
[0,0,1,1],\ [0,2,0,0],\ [0,0,2,0].
$$
Apparently the condition \eqref{eq.Nguyen-trans} is fulfilled in this case
despite $m$ is not extreme since $m= (\delta_{N} + \delta_{\{a,b,c\}})+(\delta_{N} + \delta_{\{b,c,d\}})$.
In fact, \eqref{eq.Nguyen-trans} geometrically means that $\cor (m)$
has the maximal attainable dimension $|M|-1$, where $M$ is the support of $m$. Indeed,
\eqref{eq.Nguyen-trans} means $y\in {\dv R}^{M}$ belongs to the orthogonal complement
of a~translated affine hull of $\cor (m)$.
\end{remark}

\begin{remark}\rm
The central concept in \cite{ngu78} is that of an {\em expansion\/} of an
integer-valued non-decreasing submodular game $r$ over $N$, which can be introduced as a rank function $r^{\prime}$ of
a matroid over $N^{\prime}$ such that there exists a function $\kappa:N^{\prime}\rightarrow N$ onto $N$
satisfying $r(S)=r^{\prime}(\kappa_{-1}(S))$ for any $S\subseteq N$.
Note that this is our simplified re-formulation of the definition and construction from \cite[\S\,1.3]{ngu78}.

Nguyen \cite{ngu78} shows that such a game $r$ has an expansion
$r^{\prime}$ with $|N^{\prime}|=\sum_{i\in N} r(\{ i\})$ and restricts his attention to such expansions.
Theorem 2.1.9 in \cite{ngu78} then gives an implicit ``criterion" to recognize whether an integer-valued non-decreasing submodular game
$r$ {\em is not extreme}, that is, whether it does not generate an extreme ray of the respective cone. The condition is that,
for some positive integer $k\in {\dv N}$, the multiple $k\cdot r$ has an expansion $r^{\prime}$ such that the matroid corresponding
to $r^{\prime}$ is not connected.

Despite testing of matroid connectivity is easy, the condition is not suitable as the criterion for testing
extremality of $r$. Although $|N^{\prime}|$ is fixed, there is no upper bound on the multiplicative factor
$k\in {\dv N}$, and one can hardly test an infinite number of potential rank functions $r^{\prime}$ on $N^{\prime}$
to confirm whether $r$ is extreme or not using this ``expansion criterion".
To be more specific note that it follows from the proof of \cite[Theorem~2.1.9]{ngu78} that
the above multiplicative factor $k$ is a common integer multiple of
denominators of rational coefficients in a potential non-trivial conic combination of non-decreasing submodular functions giving $r$.
\end{remark}

\subsection{Other results and summary}\label{ssec.kashi}

The 2000 paper by Kashiwabara \cite{kas00} has been inspired by Ngyuen \cite{ngu78} and
the cone of non-decreasing submodular games. A sufficient condition for extremality of an integer-valued game is offered in \cite{kas00}, which is more
general than the matroidal criterion.

Nevertheless, to tackle the extremality problem, (non-decreasing) submodular games are transformed
in \cite{kas00} by \eqref{eq.m-to-dual-r} to (non-decreasing) supermodular games and then the M\"{o}bius inversion
(see \ref{sec.apex-supermod}) is applied. The sufficient conditions
for extremality (of such an integer-valued game) are formulated in terms
of the  values of the M\"{o}bius inversion. These technical conditions from \cite[\S\,7]{kas00}
lead to the verification of certain combinatorial properties.

\begin{remark}\rm
The equivalent condition for extremality in \cite[Theorem 3.4]{kas00} seems to be analogous to the
extremality characterization in terms of conditional independence; see  \ref{sec.apex-CI},
Corollary \ref{cor.ext-ray-co-atom}. In our terms, Theorem 3.4 from \cite{kas00} says, for a non-negative
supermodular game $m$, that $m$ is extreme iff the only (non-negative standardized) game producing the
same or larger conditional independence model is a~multiple of $m$.
\end{remark}

Let us summarize the observations from \S\,\ref{sec.remark-other}.
Examples \ref{exa.RW-73-74} and \ref{exa.nguy78} show that the specific criteria discussed in \S\,\ref{ssec.conv-measure-game}
and \S\,\ref{ssec.matroid} differ from each other despite being completely analogous. They also differ from our new criterion
in \S\,\ref{ssec.main-formul} because the systems of linear constraints are different. Nevertheless,
the matroidal criterion can also be stated in term of the core as done in Corollary \ref{cor.Nguyen-trans}.
The specific criteria offered by Kashiwabara~\cite{kas00} are predominantly of combinatorial nature; they are not
formulated in terms of the core.

\section{Conclusions}\label{sec.concl}
The central topic of this paper was how to recognize whether $m\in\supmogame$ generates an extreme ray of $\supmogame$.
The reader familiar with polyhedral geometry may come up with the following suggestion.
Assuming one has at disposal the facet description of the cone, which
is our case, why not to try the following procedure. Consider a system of linear
constraints describing the smallest face $F(m)$ of the cone containing $m$ and check whether every
solution to that linear system has the form of a non-negative multiple of $m$ or not.
These constrains could be as follows: every facet containing $m$ corresponds to an equality constraint
while every other facet gives an inequality constraint.

The problem with this approach is that one has too many facets of $\supmogame$,
namely $\binom{n}{2}\cdot 2^{n-2}$, where $n=|N|$ (see \ref{sec.apex-CI}).
Thus, the complexity of such an extremality test is exponential in $n=|N|$.
On the other hand, provided one confirms our guess that $g\in\supmogame$ is a limit game
for $m$ (see Remark \ref{rem.Kuipers}) iff $g\in F(m)$, one can
perhaps utilize the main result from \cite{KVV10} to simplify the linear description of $F(m)$.

We would like to find out what is the complexity of testing extremality by means of our Theorem \ref{thm.1}.
We hope there is a chance that our linear system (a)-(b) results
in a~more efficient extremality test than the above mentioned approach.

The following open question is directly motivated by Example \ref{exa.Doker}:
is the condition from Theorem \ref{thm.1} necessary for an exact standardized game to be extreme in the respective cone?
Thus, one of our next research topics could be the cone of standardized exact games. We would like to
explicate its facet description and deal with criteria to recognize extreme exact games.
One can also study the core polytopes for extreme exact games and raise the question whether they are always indecomposable.

We consider the concept of core structure from \S\,\ref{ssec.interpret} to be of crucial significance.
One of our possible future research directions could be to search for combinatorial criteria
to test extremality of a supermodular game in terms of this concept.
However, even if such a~result is achieved, it would be just a preliminary step paving the way towards a more ambitious plan:
to achieve a {\em complete characterization\/} of extreme supermodular functions. By the complete characterization we mean
here an enumeration procedure such that, for any given $n=|N|$, the procedure generates every extreme ray of $\supmogame$.

\appendix
\section{Supermodularity}\label{sec.apex-supermod}

In this appendix, we collect various characterizations of supermodular
games appearing in the literature. The reader should be familiar with the notation and concepts from \S\,\ref{sec.prelimi}.
\smallskip

Let $\mu_{m}$ denote the {\em M\"{o}bius inversion\/} of a game $m$ over $N$, that is,
$$
\mu_{m}(A) ~:=~ \sum_{B\subseteq A}\, (-1)^{|A\setminus B|}\cdot m(B)\qquad \mbox{for every $A\subseteq N$.}
$$
The embedding of $\caP$ into ${\dv R}^{N}$ by means of the (incidence vector)
mapping $S\mapsto \incS_{S}$, for $S\subseteq N$, allows one to interpret any game $m$ over $N$ as a real function on $\{ 0,1\}^{N}$. Indeed, it suffices to put $\widehat{m}(\incS_{S}):= m(S)$ for any $S\subseteq N$. Below we describe a natural way of extending $\widehat{m}$ to all nonnegative vectors in ${\dv R}^{N}$;
we are going to denote this extension by $\widehat{m}$ as well because there is no danger of confusion in this paper.

The idea is that, for any $y\in [0,\infty )^{N}$, there is a unique chain\footnote{Here, by a {\em chain\/} is meant a class of sets
${\cal C}\subseteq\caP$ such that $\forall\, A,B\in {\cal C}$ either $A\subseteq B$ or $B\subseteq A$.}
${\cal C}^{y}$ of subsets of $N$ such that $N\in {\cal C}^{y}$, $\emptyset\not\in {\cal C}^{y}$ and there are unique
coefficients $\lambda_{N}\geq 0$ and $\lambda_{S}>0$ for $S\in {\cal C}^{y}\setminus\{ N\}$ such that
\begin{equation}
y=  \sum_{S\in {\cal C}^{y}}\, \lambda_{S}\cdot \incS_{S}\,.
\label{eq:canonical}
\end{equation}
The {\em Lov\'{a}sz extension $\widehat{m}: [0,\infty )^{N}\to {\dv R}$} (of $m$) is then defined linearly with respect to the decomposition \eqref{eq:canonical}, that is,
$$
\widehat{m}(y) :=  \sum_{S\in {\cal C}^{y}}\, \lambda_{S}\cdot \widehat{m}(\incS_{S}) \equiv \sum_{S\in {\cal C}^{y}}\, \lambda_{S}\cdot m(S)\qquad
\mbox{for every $y\in [0,\infty )^{N}$}\,.
$$
Some researchers also call $\widehat{m}(y)$ the ({\em discrete}) {\em Choquet integral\/} of $y$ with
respect to $m$ \cite{GrabischHandbook}. Note that, for any maximal chain ${\cal C}_{\pi}$, $\pi\in\Upsilon$ as introduced in \eqref{eq.max-chain},
$\widehat{m}$ is linear on the cone spanned by the vectors $\incS_{S}$, $S\in {\cal C}_{\pi}$. Indeed, realize that the vectors $\incS_{S}$,
$S\in {\cal C}_{\pi}\setminus\{\emptyset\}$ are linearly independent and $y$ in the
cone has a unique decomposition
$y=\sum_{S\in {\cal C}_{\pi}\setminus\{\emptyset\}} \lambda_{S}\cdot\incS_{S}$
(with $\lambda_{S}\geq 0$). Dropping some zero coefficients then leads to~\eqref{eq:canonical}.
\smallskip

This implies that the following properties are true.
\begin{itemize}
\item The function $\widehat{m}$ is continuous and piecewise linear on $[0,\infty )^{N}$.
\item $\widehat{m}(\lambda\cdot y)=\lambda\cdot \widehat{m}(y)$\quad
for every $\lambda \geq 0$ and $y\in[0,\infty)^{N}$.
\item $\widehat{m_{1}+m_{2}}=\widehat{m_{1}}+\widehat{m_{2}}$\quad for every pair of
games $m_{1},m_{2}$ over $N$.
\item $\widehat{u\cdot m}=u\cdot\widehat{m}$\quad for every game $m$ over $N$ and every $u\in {\dv R}$.
\end{itemize}
For a real number $u\in {\dv R}$, we abbreviate $u^{+}:=\max\,\{u,0\}$ and $u^{-}:=\max\,\{-u,0\}$.
The {\em upper core of $m$} (see \cite[\S\,4]{DK00}) is the polytope
$$
\cor^{+}(m) ~:=~ \bigoplus_{\substack{S\subseteq N\\ \mu_{m}(S)>0}}  \mu_{m}(S)^{+} \cdot\Delta_{S}\,,
$$
where $\bigoplus$ denotes the multiple Minkowski sum. Analogously, the {\em lower core of $m$} is $$
\cor^{-}(m) ~:=~ \bigoplus_{\substack{S\subseteq N\\ \mu_{m}(S)<0}}  \mu_{m}(S)^{-} \cdot\Delta_{S}.
$$

\begin{thm}\label{thm:super}
Given a game $m$ over $N$, the following conditions are equivalent.
\begin{enumerate}
\renewcommand{\theenumi}{\roman{enumi}}
\renewcommand{\labelenumi}{(\theenumi)}
\item\label{i:supermod} $m$ is supermodular.
\item\label{i:convex-max} For every $A\subseteq B\subseteq N$ and every $C\subseteq N\setminus B$, one has
$$
m(A\cup C)-m(A)\leq m(B\cup C)-m(B).
$$
\item\label{i:convex-mid} For every $A\subseteq B\subseteq N$ and every $i\in N\setminus B$, one has
$$
m(A\cup \{i\})-m(A)\leq m(B\cup \{i\})-m(B).
$$
\item\label{i:convex-min} For every $i,j\in N$ with $i\neq j$ and every $A\subseteq N\setminus\{i,j\}$, one has
$$
m(A\cup\{i\})-m(A)\leq m(A\cup \{i,j\})-m(A\cup\{j\}).
$$
\item\label{i:Moeb} The M\"{o}bius inversion $\mu_{m}$ satisfies
$$
\forall\, i,j\in N,\ i\neq j\quad \forall A\subseteq N\setminus\{ i,j\}\qquad \sum_{B\subseteq A} \mu_{m}(B\cup\{i,j\}) \geq 0.
$$

\item\label{i:Moeb2} The M\"{o}bius inversion $\mu_{m}$ satisfies for
every $A,B\subseteq N$~
$$
\sum_{D\in {\cal D}}\, \mu_{m}(D) \geq 0\quad
\mbox{where~ ${\cal D}=\{\, D\subseteq (A\cup B)\,:\ D\setminus B\neq\emptyset ~~\&~~
D\setminus A\neq\emptyset\,\}$}.
$$

\item\label{i:margcore} For each $\pi\in\Upsilon$, one has $x^{m}(\pi,\ast)\in \cor(m)$; in other words,
all marginal vectors of $m$ belong to the core of $m$.
\item\label{i:facelat} $\cor (m)\neq\emptyset$ and, for every $S,T\subseteq N$,
$$
F_{S}(m) \cap F_{T}(m)  \subseteq F_{S\cup T}(m) \cap F_{S\cap T}(m),
$$
where $F_{S}(m):=\{\, v\in\cor(m)\,:\ \sum_{i\in S} v_{i}=m(S)\,\}$
is the face of $\cor (m)$ for $S\subseteq N$.
\item\label{i:tesnoset}  $\cor (m)\neq\emptyset$ and, for every $v\in \cor(m)$,
the class of tightness sets
$$
{\cal S}^{m}_{v} ~:=~ \{\, S\subseteq N\,:\ \sum_{i\in S} v_{i}=m(S)\,\}
$$
is closed under the operations of intersection and union.\footnote{Another formulation of \eqref{i:tesnoset}
is that ${\cal S}^{m}_{v}$ is a lattice relative to $\cap$ and $\cup$, for any $v\in\cor (m)$.}
\item\label{i:W=C} $\cor(m)=\web (m)$.
\item\label{i:perm-min}
For every $S\subseteq N$, one has $m(S)=\min\limits_{\tau\in\Upsilon}\, \sum_{i\in S}\, x^{m}(\tau ,i)$\,.
\item\label{i:web-min} For every $S\subseteq N$, one has
$m(S)=\min\limits_{v\in \web (m)}\, \sum_{i\in S} v_{i}$\,.
\item\label{i:conc} The Lov\'{a}sz extension $\widehat{m}$ of $m$ is a concave real function on $[0,\infty )^{N}$.
\item\label{i:lsf} For every $y\in [0,\infty )^{N}$ one has\,
$\widehat{m}(y)= \min\limits_{v\in\web (m)}\, \sum_{j\in N} v(j)\cdot y(j)$.
\item\label{i:Mink} $\cor^{+}(m)=\cor(m)\oplus \cor^{-}(m)$.
\end{enumerate}
\end{thm}
\noindent
Observe that the crucial Lemma \ref{lem.1} now follows from Theorem \ref{thm:super}, which says \eqref{i:supermod} $\Leftrightarrow$ \eqref{i:perm-min}\/
and, moreover, \eqref{i:supermod} $\Rightarrow$ \eqref{i:W=C}.

\begin{proof}
The equivalence \eqref{i:supermod} $\Leftrightarrow$ \eqref{i:convex-max} $\Leftrightarrow$ \eqref{i:convex-mid} $\Leftrightarrow$ \eqref{i:margcore}
was shown in 1981 note by Ichiishi \cite{ich81}; for \eqref{i:margcore} $\Leftrightarrow$ \eqref{i:convex-min} see  \cite[Corollary~8]{KVV10};
\eqref{i:supermod} $\Leftrightarrow$ \eqref{i:Moeb} was shown in \cite[Theorem~9]{KVV10}; \eqref{i:supermod} $\Leftrightarrow$ \eqref{i:Moeb2} appeared in \cite[Corollary~2]{CJ89}.
Hence, \eqref{i:supermod} $\Leftrightarrow\ldots\Leftrightarrow$ \eqref{i:margcore}.

Theorem 5 in 1972 paper by Shapley \cite{sha72} implies \eqref{i:supermod} $\Rightarrow$ \eqref{i:facelat}. Clearly, \eqref{i:tesnoset} is an equivalent formulation of \eqref{i:facelat}. Theorem 3 in \cite{sha72} says that \eqref{i:facelat}\/ implies that the extreme points of $\cor (m)$
are precisely the marginal vectors of $m$; in particular, \eqref{i:tesnoset} $\Leftrightarrow$ \eqref{i:facelat} $\Rightarrow$ \eqref{i:W=C}.
The implication \eqref{i:W=C} $\Rightarrow$ \eqref{i:margcore} is evident.
Hence, \eqref{i:supermod}~$\Leftrightarrow\ldots\Leftrightarrow$~\eqref{i:W=C}.

To show \eqref{i:margcore} $\Rightarrow$ \eqref{i:perm-min} consider fixed $S\subseteq N$ and
the fact $x^{m}(\tau ,\ast)\in\cor (m)$  for any $\tau\in\Upsilon$ gives
$\sum_{i\in S}\, x^{m}(\tau ,i)\geq m(S)$ implying
$\min_{\tau\in\Upsilon}\, \sum_{i\in S}\, x^{m}(\tau ,i)\geq m(S)$. Conversely,
there exists $\pi\in\Upsilon$ with $S\in {\cal C}_{\pi}$ and, by \eqref{eq.chain},
$m(S)=\sum_{i\in S}\, x^{m}(\pi ,i)\geq \min_{\tau\in\Upsilon}\, \sum_{i\in S}\, x^{m}(\tau ,i)$.
In order to show \eqref{i:perm-min} $\Rightarrow$ \eqref{i:margcore},
consider fixed $\pi\in\Upsilon$ and the marginal vector $x^{m}(\pi ,\ast)$. Write for $S\subseteq N$,
$$
\sum_{i\in S}\, x^{m}(\pi ,i) ~\geq ~\min\limits_{\tau\in\Upsilon}\, \sum_{i\in S}\, x^{m}(\tau ,i) = m(S)\,.
$$
For $S=N$ one has $\sum_{i\in N}\, x^{m}(\pi ,i)=m(N)$ directly from \eqref{eq.payoff-array}.
Thus, $x^{m}(\pi ,\ast)\in \cor(m)$.
By the definition of $\web (m)$, \eqref{i:perm-min} $\Leftrightarrow$ \eqref{i:web-min}.
Hence, \eqref{i:supermod} $\Leftrightarrow\ldots\Leftrightarrow$ \eqref{i:web-min}.

The equivalence \eqref{i:supermod} $\Leftrightarrow$ \eqref{i:conc} is the ``supermodular version''
of the theorem originally proved by Lov\'{a}sz for submodular functions \cite[\S\,4]{lov83}; see also \cite[Theorem~6.13]{fuj91}. The equivalences \eqref{i:supermod} $\Leftrightarrow$ \eqref{i:lsf} $\Leftrightarrow$ \eqref{i:Mink} were shown
by Danilov and Koshevoy \cite[\S\,5]{DK00}. Hence, we can conclude that
\eqref{i:supermod} $\Leftrightarrow\ldots\Leftrightarrow$ \eqref{i:Mink}.
\end{proof}

\begin{remark}\rm\label{rem:coopconvex}
In cooperative game theory, \eqref{i:supermod} $\Leftrightarrow$ \eqref{i:convex-mid} is interpreted that a function
$m$ is supermodular iff the marginal contribution of a player to a coalition is monotone non-decreasing
with respect to set-theoretical inclusion. This explains the term ``convex'' used to name supermodular functions
in game theory in the analogy with one of equivalent characterizations of convexity of a real function.
Observe that this established terminology contrasts with the concavity of the Lov\'{a}sz extension;
see also our discussion in \S\,\ref{Rosen-compare}, in particular, Remark~\ref{rem.RW-misuse}.
\end{remark}

The following observation also follows from Theorem \ref{thm:super}.

\begin{coro}\rm\label{cor.zero-one-core}
Assuming $m\in\supmogame$,
all the vertices of the core $\cor (m)$ are zero-one vectors
iff $m$ is integer-valued and $m(N)-m(N\setminus\{ i\})\leq 1$ for any $i\in N$.
\end{coro}

\begin{proof}
For necessity use Theorem \ref{thm:super}, conditions \eqref{i:web-min}
and \eqref{i:W=C}, to derive that, for any $S\subseteq N$, $m(S)=\min_{v\in {\cal X}}\, \sum_{i\in S} v_{i}$,
where ${\cal X}=\ext (\cor (m))$. Since ${\cal X}\subseteq \{ 0,1\}^{N}$, $m$ is integer-valued.
For any $i\in N$, it also implies that $v\in {\cal X}$ with
$m(N\setminus\{ i\})=\sum_{j\in N\setminus\{ i\}} v_{j}$ exists. Thus,
$m(N)-m(N\setminus\{ i\})=\sum_{j\in N} v_{j} - \sum_{j\in N\setminus\{ i\}} v_{j} = v_{i}\leq 1$.

For sufficiency use Theorem \ref{thm:super}, condition \eqref{i:W=C},
and the definition of $\web (m)$, to observe that every vertex of $\cor (m)$ is a row in
\eqref{eq.payoff-array}, and, therefore, has integers as components.
For every $v\in\cor (m)$ and $i\in N$, by \eqref{eq.def.core},
$v_{i}=\sum_{j\in N} v_{j} - \sum_{j\in N\setminus\{ i\}} v_{j}\leq m(N)-m(N\setminus\{ i\})\leq 1$
and $v_{i}\geq m(\{ i\})=0$ implying $\cor (m)\subseteq [0,1]^{N}$. Altogether,
$\ext (\cor (m))\subseteq \{ 0,1\}^{N}$.
\end{proof}

\section{Conditional independence (CI) interpretation}\label{sec.apex-CI}
Given a supermodular set function $m$ over $N$ and pairwise disjoint subsets $X,Y,Z\subseteq N$,
we say that $X$ is {\em conditionally independent of\/ $Y$ given $Z$ with respect to $m$} and write
\begin{equation}
X\ci Y\,|\,Z\,\,[m]\qquad \mbox{iff}\quad m(X\cup Y\cup Z) +m(Z) -m(X\cup Z) - m(Y\cup Z)=0\,.
\label{eq.def-CI}
\end{equation}
The statement $X\ci Y\,|\,Z\,\,[m]$ is then called the {\em conditional independence} (CI) {\em statement}.
The {\em CI model produced by $m$} then consists of valid CI statements with respect to $m$.
The collection of {\em structural independence models} over $N$, introduced in \cite[\S\,5.4.2]{stu05},
can equivalently be defined as the class of CI models produced by supermodular set functions over $N$. This collection is a finite lattice whose order is given by the set-theoretic inclusion
between classes of represented CI statements.

\begin{remark}\rm\label{rem.why-variable}
The concept of a structural independence model
generalizes the concept of a probabilistic CI structure; see \cite[\S\,5.1.1]{stu05} for detailed explanation.
In the context of probabilistic CI, the elements of $N$ correspond to {\em random variables}, usually finite-valued ones.
The probabilistic CI statement $X\ci Y\,|\,Z$ then means that the (set of random) variables in $X$ is
stochastically independent of the variables in $Y$ conditionally on (the values of) the variables in $Z$.
This interpretation was our motivation to call the elements of $N$ {\em variables} in this paper.
\end{remark}

Let $\eletri$ denote the class of all triplets $\langle a,b|Z\rangle$,
where $a,b\in N$ are distinct and $Z\subseteq N\setminus\{a,b\}$. For each such triplet and function $m\in {\dv R}^{\caP}$, put
$$
\Delta m(a,b|Z) ~:=~ m(\{ a,b\}\cup Z)+m(Z)-m(\{ a\}\cup Z)-m(\{ b\}\cup Z).
$$
By Theorem \ref{thm:super}\eqref{i:convex-min}, the expression $\Delta m(a,b|Z)$
is always non-negative for a supermodular function $m$. In~fact, one even has
$m(X\cup Y\cup Z) +m(Z) -m(X\cup Z) - m(Y\cup Z)\geq 0$
for any triplet $X,Y,Z$ of pairwise disjoint subsets of $N$.
Hence, any structural independence model is a \emph{semi-graphoid}, which
is a concept proposed by Pearl \cite{pea98}:
for any pairwise disjoint $X,Y,Z,W\subseteq N$, one has $\emptyset\ci Y\,|\,Z$ and
\begin{eqnarray*}
X\ci Y\,|\,Z &\Leftrightarrow & Y\ci X\,|\,Z\,,\\
X\ci Y\cup W\,|\,Z &\Leftrightarrow & \{\, X\ci Y\,|\,Z\cup W ~~\&~~
X\ci W\,|\,Z \,\}\,.
\end{eqnarray*}
This implies that every structural model is determined by its {\em elementary CI statements},
which are statements of the form $\{a\}\ci \{b\}\,|\,Z$ where $a,b\in N$, $a\neq b$ and
$Z\subseteq N\setminus\{a,b\}$. By Theorem \ref{thm:super}\eqref{i:convex-min}, the class of supermodular games $\supermo$
is a (rational) polyhedral cone in ${\dv R}^{\caP}$ characterized by $\binom{n}{2}\cdot 2^{n-2}$ inequalities as follows:
$$
m\in\supermo ~\Leftrightarrow ~
\left[\,\forall\,\langle a,b|Z\rangle\in\eletri \quad \Delta m(a,b|Z)\geq 0 \,\right]\quad
\mbox{for $m\in {\dv R}^{\caP}$ with $m(\emptyset)=0$.}
$$
A well-known fact is that the inequalities above are exactly the facet-defining inequalities for $\supermo$
and its standardized version $\supmogame$; see, for example,  \cite[Corollary 11]{KVV10} or one can derive that from \cite[Lemma 2.1]{KSTT12}.

We will say that functions $m^{1},m^{2}\in\supermo$ are {\em qualitatively equivalent} (see \cite[\S\,5.1.1]{stu05})
and write $m^{1}\sim m^{2}$ if they produce the same CI model. It follows from the
semi-graphoid properties mentioned above
that $m^{1}\sim m^{2}$ iff\/ ${\cal I}(m^{1})={\cal I}(m^{2})$, where
$$
{\cal I}(m) ~:=~ \{\, \langle a,b|Z\rangle\in\eletri \,:\ \Delta m(a,b|Z)=0\,\}.
$$
Let's denote by $\facelatt$ the lattice of {\em non-empty faces\/} of\/ $\supermo$
ordered by inclusion $\subseteq$, which is, of course,
isomorphic to the lattice of non-empty faces of $\supmogame$.
For any $m\in\supermo$, $F(m)$ will denote the smallest face containing $m$:
$$
F(m) := \bigcap\, \{ F :\ F\in\facelatt ~\& ~ m\in F\}\,.
$$
The following lemma implies that the face lattice $\facelatt$  of $\supermo$ is anti-isomorphic to the lattice of structural models.

\begin{lem}\rm\label{lem.face}
~~$\forall\,  m^{1},m^{2}\in\supermo\qquad
F(m^{1})\subseteq F(m^{2}) ~\Leftrightarrow~ {\cal I}(m^{1})\supseteq {\cal I}(m^{2})$\,.
\end{lem}

\begin{proof}
Let ${\cal F}_{*}(N)$ denote the class of facets of\/ $\supermo$. As every face is the intersection of facets containing it, one has $F(m)=\bigcap_{F\in{\cal F}_{*}(N) ,\, m\in F} F$ for every $m\in\supermo$. Here, the whole cone $\supermo$ is conventionally the intersection of the empty collection
of facets. Hence, because $m\in F(m)$ for any $m\in\supermo$, $F(m^{1})\subseteq F(m^{2})$ iff
\begin{equation}
\forall\, F\in{\cal F}_{*}(N)\qquad m^{2}\in F ~\Rightarrow~  m^{1}\in F\,.
\label{eq.face}
\end{equation}
However, the facets of\/ $\supermo$ correspond to triplets in $\eletri$, more specifically, they have the form
$F=\{ m\in\supermo :~ \Delta m(a,b|Z) =0\,\}$ for $\langle a,b|Z\rangle\in\eletri$.
Thus, \eqref{eq.face} is equivalent to $\forall\,\langle a,b|Z\rangle\in\eletri$,
$\Delta m^{2}(a,b|Z) =0 ~\Rightarrow~ \Delta m^{1}(a,b|Z)=0$, which is nothing but
${\cal I}(m^{1})\supseteq {\cal I}(m^{2})$.
\end{proof}

\begin{coro}\rm\label{cor.face}
~~$\forall\,  m^{1},m^{2}\in\supermo\quad m^{1}\sim m^{2}  ~\Leftrightarrow~
{\cal I}(m^{1})={\cal I}(m^{2}) ~\Leftrightarrow~ F(m^{1})=F(m^{2})$.\\
Hence, the equivalence classes of $\sim$ are relative interiors of faces of $\supermo$. In other words,
$$
\forall\,  m^{1},m^{2}\in\supermo\qquad m^{1}\sim m^{2}  ~\Leftrightarrow~
\left[ \exists\, F\in\facelatt \quad m^{1},m^{2}\in\inter (F)\,\right].
$$
\end{coro}

\begin{proof}
For $F\in\facelatt$, $\inter (F)$ is the set of all $m\in\supermo$ such that
$F(m)=F$.
\end{proof}

Another consequence of Lemma \ref{lem.face} is the following observation,
which has already been mentioned as Lemma 5.6 in \cite{stu05}.

\begin{coro}\rm\label{cor.ext-ray-co-atom} ~~
A game $m$ generates an extreme ray of $\supmogame$ iff
the CI model produced by $m$ is a co-atom in the lattice of structural independence models,
which means that the only structural model strictly containing it is the complete independence model.\footnote{The complete independence
model has\/ $\eletri$ as the set of valid elementary independence statements.}
\end{coro}

\begin{remark}\rm\label{rem.CI-interpret}
The class of structural independence models can alternatively be introduced in terms of
{\em submodular\/} set functions. This corresponds to the description of a probabilistic CI
structure by means of the {\em entropy function}; see Remark 4.4 in \cite{stu05}.
Specifically, given a~submodular game $r$ over $N$, $\{\, \langle a,b|Z\rangle\in\eletri \,:\ \Delta r(a,b|Z)=0\,\}$
is the respective structural independence model. Since $\Delta r(a,b|Z)\leq 0$ for any such $r$ and $\langle a,b|Z\rangle\in\eletri$
everything works like in the supermodular case.  In particular, the correspondence $r\leftrightarrow m:=-r$ is the correspondence which preserves CI interpretation.
\end{remark}

\section*{Acknowledgements}
The work on this topic has been supported from the GA\v{C}R grant project n.\ 13-20012S. Tom\'{a}\v{s} Kroupa gratefully acknowledges partial support from Marie Curie Intra-European Fellowship OASIG (PIEF-GA-2013-622645). We are indebted to our colleague Fero Mat\'{u}\v{s} for pointing our attention to a highly relevant paper \cite{ngu78}
and to Jasper De Bock for references to \cite{mey74,QdC08}.
The package CONVEX for Maple by Matthias Franz \cite{Franz09} helped us in processing the cores and their vertices.

\end{document}